\documentclass[a4paper,11pt,reqno]{article}

\usepackage[utf8]{inputenc}
\usepackage[T1]{fontenc}
\usepackage{lmodern}
\usepackage[english]{babel}
\usepackage{microtype}

\usepackage{amsmath,amssymb,amsfonts,amsthm}
\usepackage{mathtools,accents}
\usepackage{mathrsfs}
\usepackage{aliascnt}
\usepackage{braket}
\usepackage{bm}

\usepackage[a4paper,margin=3cm]{geometry}
\usepackage[citecolor=blue,colorlinks]{hyperref}

\usepackage{enumerate}
\usepackage{xcolor}

\makeatletter
\g@addto@macro\@floatboxreset\centering
\makeatother

\makeatletter
\def\newaliasedtheorem#1[#2]#3{
  \newaliascnt{#1@alt}{#2}
  \newtheorem{#1}[#1@alt]{#3}
  \expandafter\newcommand\csname #1@altname\endcsname{#3}
}
\makeatother

\numberwithin{equation}{section}

\newtheoremstyle{slanted}{\topsep}{\topsep}{\slshape}{}{\bfseries}{.}{.5em}{}

\theoremstyle{plain}
\newtheorem{theorem}{Theorem}[section]
\newaliasedtheorem{proposition}[theorem]{Proposition}
\newaliasedtheorem{lemma}[theorem]{Lemma}
\newaliasedtheorem{corollary}[theorem]{Corollary}
\newaliasedtheorem{conjecture}[theorem]{Conjecture}
\newaliasedtheorem{counterexample}[theorem]{Counterexample}

\theoremstyle{definition}
\newaliasedtheorem{definition}[theorem]{Definition}
\newaliasedtheorem{question}[theorem]{Question}
\newaliasedtheorem{example}[theorem]{Example}

\theoremstyle{remark}
\newaliasedtheorem{remark}[theorem]{Remark}

\let\altphi\phi
\let\phi\varphi
\let\varphi\altphi
\let\altphi\undefined

\newcommand{\id}{\mathrm{id}}

\newcommand{\di}{\mathop{}\!\mathrm{d}}

\newfont{\tmpf}{cmsy10 scaled 2500}

\newcommand{\intav}{{\mathop{\int\kern-10pt\rotatebox{0}{\textbf{--}}}}}

\renewcommand{\ }{\text{ }}
\def\<{\langle}
\def\>{\rangle}

\begin{document}
\title{Asymptotic behavior of solutions to linear evolution equations with time delay via a spectral theory on Gelfand triples}
\author{
Haozhe Shu
\thanks{Mathematical Institute, Tohoku University, Sendai, 980-8578, Japan / Advanced Institute for Material Research, Tohoku University, Sendai, 980-8577, Japan; \url{shu.haozhe.t7@dc.tohoku.ac.jp }}} 
\maketitle
\begin{abstract}
In this paper, a class of linear evolution equations with time delay is studied in which the presence of continuous spectrum on the imaginary axis obstructs the analysis of long-time dynamics. To address it, a generalized spectral framework on a Gelfand triple is utilized. When the spectral measure of the unperturbed term (a skew-adjoint operator) admits some analyticity condition, the resolvent is extended to a generalized resolvent. Called generalized spectrum, the collection of singularities on the Riemann surface of the generalized resolvent may differ from the spectrum in the usual sense because of the change of topology via the Gelfand triple. 
It is shown that under some compactness assumption, the generalized spectrum consists only of isolated generalized eigenvalues (resonance poles). This structure allows contour deformation in the inverse Laplace representation and yields exponential decay in a weak topology. As an application, we analyze the continuum limit of the Kuramoto-Daido model with time delay and prove linear stability of the incoherent state in the weak coupling regime.
\end{abstract}

\section{Introduction} 
Let $\mathcal{H}$ be a separable Hilbert space and let $\tau>0$. We consider the linear evolution equation with a single time delay
\begin{equation}\label{eq:sec1_linear evolution}
    \partial_t u(t)=iHu(t)+Ku(t-\tau),
\end{equation}
where $H$ is an (unbounded) self-adjoint operator on $\mathcal{H}$ and $K$ is bounded. $u:[-\tau,+\infty)\to \mathcal{H}$. When $\tau=0$, the equation reduces to a bounded perturbation of the skew-adjoint generator $iH$. The presence of delay introduces a nonlocal effect in time and leads to a transcendental characteristic equation. Our objective is to analyze the long-time behavior of solutions and to clarify how the delay term modifies the spectral structure of the associated generator.

The long-time behavior of linear evolution equation $\frac{\di}{\di t}x=Ax$ is closely related to the spectral properties of the associated generator $A$. When $A$ generates an analytic semigroup, the growth bound of the semigroup coincides with the spectral bound $s(A):=\sup\{\Re(\lambda)|\lambda\in \sigma(A)\}$. In general, however, this identity fails for $C_0$ semigroups. The Gearhart-Prüss theorem \cite{nagel} shows that exponential stability requires uniform boundedness of the resolvent on the imaginary axis, namely
$$\sup_{\lambda\in  i\mathbb{R}}||(\lambda-A)^{-1}||<\infty.$$

Difficulties arise when the continuous spectrum of $A$ intersects the imaginary axis. In this case, the resolvent cannot remain uniformly bounded near the spectrum, and the classical exponential stability criterion is no longer applicable. Nevertheless, solutions may still exhibit decay in a weak topology. Such weak decay phenomena are known as Landau damping in plasma physics \cite{ch}. The corresponding decay rates are often described in terms of resonance poles \cite{reed,rauch,gp} or zeros of the Evans function \cite{zumbrun}.

To illustrate how continuous spectrum arises in the delay setting, we rewrite (\ref{eq:sec1_linear evolution}) as an abstract Cauchy problem on the product Hilbert space $\mathcal{H}\times L^2([-\tau,0];\mathcal{H})$, following \cite{batkai1,batkai2}. There is a one-to-one correspondence between solutions of (\ref{eq:sec1_linear evolution}) and solutions of
\begin{equation}
    \begin{aligned}
        \frac{\di}{\di t}\mathcal{U}(t)=\mathcal{A}\mathcal{U}(t):=
        \begin{pmatrix}
             iH & \hat{K}\\
            0 & \frac{\di}{\di s}
        \end{pmatrix}
        \mathcal{U}(t),
    \end{aligned}
\end{equation}
where $\mathcal{U}:[-\tau,+\infty)\to \mathcal{H}\times L^2([-\tau,0];\mathcal{H})$ and $\hat{K}:L^2([-\tau,0];\mathcal{H})\to \mathcal{H}$ is defined by $\hat{K}f:=Kf(-\tau)$.

Solving $(\lambda-\mathcal{A})(x,f)^\top=(y,g)^\top$ leads to the characteristic equation
\begin{equation}\label{eq:sec1_characteristic equation}
    \Delta(\lambda)x=(\lambda- iH-e^{-\lambda\tau}K)x=0.
\end{equation}
Moreover, the resolvent of $\mathcal{A}$ is given by
\begin{equation}
    R(\lambda;\mathcal{A})=\begin{pmatrix}
        \Delta(\lambda)^{-1} & \Delta(\lambda)^{-1}\hat{K}(\lambda-A_L)^{-1})\\
        \varepsilon_\lambda\otimes \Delta(\lambda)^{-1} & [\varepsilon_\lambda\otimes \Delta(\lambda)^{-1}+\id](\lambda-A_L)^{-1}
    \end{pmatrix}.
\end{equation}
A detailed explanation of the notation will be given in Section~\ref{sec:2}. This presentation shows that the spectrum of $\mathcal{A}$ coincides with the set of singularities of the operator-valued function $\Delta(\lambda)^{-1}$ (see \cite{batkai1}). We refer to $\Delta(\lambda)^{-1}:\mathcal{H}\to\mathcal{H}$ as the retarded resolvent associated with the delay equation, following \cite{nakagiri2}.

Applying the inverse Laplace transform, solutions of the original delay equation (\ref{eq:sec1_linear evolution}) can be represented as 
\begin{equation}\label{eq:sec1_inverse laplace}
    u(t)=\frac{1}{2\pi i}\int_\Gamma e^{\lambda t}\Delta(\lambda)^{-1}x\di \lambda+\frac{1}{2\pi i}\int_\Gamma e^{\lambda t}\Delta(\lambda)^{-1}\hat{K}(\lambda-A_L)^{-1}f\di \lambda,
\end{equation}
where $(x,f)$ corresponds to the initial data. The contour $\Gamma$ is taken to be the vertical line $\{\lambda\mid \Re(\lambda)=a\}$ with $a$ sufficiently large. The difficulty in analyzing the long-time behavior lies in the fact that inherited from the continuous spectrum of $iH$, continuous singularities of $\Delta(\lambda)^{-1}$ may exist on the imaginary axis. As a consequence, direct contour deformation in the inverse Laplace representation (\ref{eq:sec1_inverse laplace}) is obstructed.

In the present paper, construction of a \textit{suitable} Gelfand triple (Rigged Hilbert space) $X\subset\mathcal{H}\subset X^\prime$ plays a central role in overcoming this difficulty caused by continuous spectrum. Here $X$ is a dense subspace of $\mathcal{H}$ endowed with a stronger topology and $X^\prime$ denotes its topological dual. The inclusion map $\kappa:X\to X^\prime$ is continuous and has dense range. The motivation for introducing the Gelfand triple is that the spectral properties of an operator may change under a modification of topology. Following a theoretic framework of Chiba \cite{chibaA,chibaB}, we show that under proper analyticity conditions, the retarded resolvent $\Delta(\lambda)^{-1}$ admits an extension as an operator from $X$ to $X^\prime$, denoted by
\begin{equation}
    \mathcal{R}_\tau(\lambda)\circ\kappa:X\to X^\prime.
\end{equation}
We refer to $\mathcal{R}_\tau(\lambda)$ as the generalized retarded resolvent. The generalized spectrum $\hat{\sigma}(\mathcal{A})$ is defined as the set of singularities of $\mathcal{R}_\tau(\lambda)\kappa$. Since $X$ has a stronger topology than $\mathcal{H}$ and $X^\prime$ has a weaker topology, continuous spectrum of $\mathcal{A}$ may disappear in the generalized sense. More precisely, it is shown that under some additional compactness assumptions on the perturbation term $K$, the generalized spectrum consists of only isolated generalized eigenvalues (resonance poles). In this context, by replacing $\Delta(\lambda)^{-1}$ by $\mathcal{R}_\tau(\lambda)\circ \kappa$ in (\ref{eq:sec1_inverse laplace}), deformation of integral contour can be implemented traced along a Riemann surface induced by the Gelfand triple. Fig.~\ref{fig:1} is an illustration for this procedure. Here $u(t)$ takes values in $X^\prime$ and the integrals (known as weak$^\ast$ Pettis integrals or Gelfand integrals) converge in dual topology.
\begin{figure}
    \centering
    \includegraphics[width=0.5\linewidth]{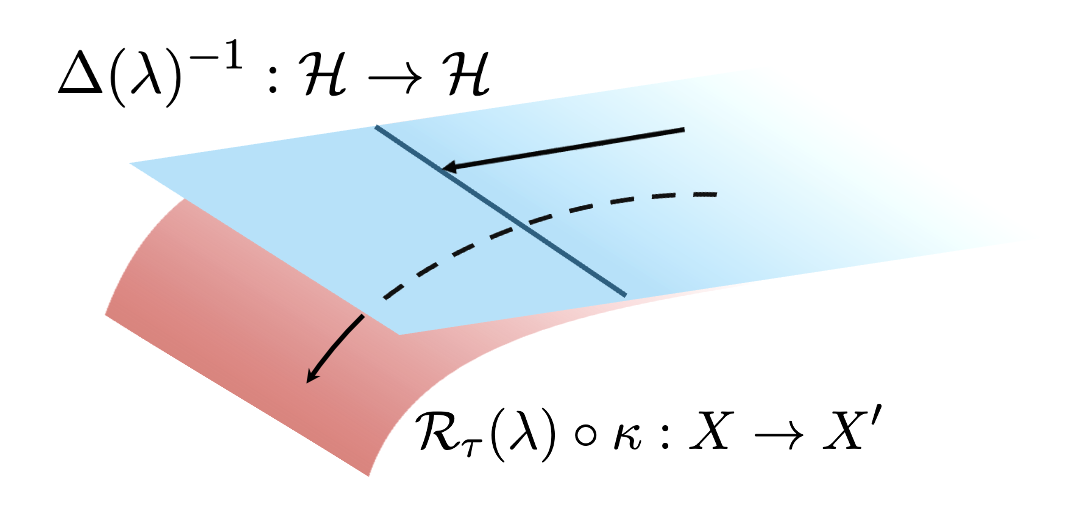}
    \caption{In the classical sense, the contour cannot be deformed to the left complex half-plane due to existence of continuous spectrum (blue line). By extending the retarded resolvent via a Gelfand triple, continuous spectrum disappear from a new Riemann surface. Analytic contour deformation is applicable where resonance poles of $\mathcal{R_\tau}(\lambda)\kappa$ on the second Riemann sheet explains asymptotic behavior of weak solutions.}
    \label{fig:1}
\end{figure}
\subsection{Main assumptions}
\noindent\textbf{Analyticity assumption.}
In this paper, proper assumptions are imposed to the Gelfand triple $X\subset \mathcal{H}\subset X^\prime$ such that the resolvent of the unperturbated term $iH$, $(\lambda-iH)^{-1}$ admits an analytic continuation across the branch cut $i\mathbb{R}$ to the left complex half-plane. The continuation, denoted by $A(\lambda):\kappa X\to X^\prime$, is represented as 
\begin{equation}\label{eq:sec1_A(lambda)}
    \langle A(\lambda)\kappa(\psi)|\varphi\rangle =\begin{cases}
        \int_{-\infty}^{+\infty}\frac{1}{\lambda-i\omega}E[\psi,\varphi](\omega)\di \omega+2\pi E[\psi,\varphi](-i\lambda),\quad \Re(\lambda)<0,\\
        \lim_{\Re(\lambda)\to 0+}\int_{-\infty}^{+\infty}\frac{1}{\lambda-i\omega}E[\psi,\varphi](\omega)\di \omega,\quad \Re(\lambda)=0,\\
        \int_{-\infty}^{+\infty}\frac{1}{\lambda-i\omega}E[\psi,\varphi](\omega)\di \omega,\quad \Re(\lambda)>0,
    \end{cases}
\end{equation}
for each $\psi,\varphi\in X$. Here, $E[\psi,\varphi](\omega):=\frac{\di}{\di \omega}(E(\omega)\psi,\varphi)$ is the density function of the spectral measure of $H$. It is clear that 
$$
\langle A(\lambda)\kappa(\psi)|\varphi\rangle= ((\lambda-iH)^{-1}\psi,\varphi)
$$
on the right half-plane. 

When the spectral measure $E(\omega)$ is additionally assumed to be absolutely continuous, an injective factor $\lambda-iH$ can be taken out from the characteristic equation (\ref{eq:sec1_characteristic equation}). By replacing $(\lambda-iH)^{-1}$ by its analytic continuation $A(\lambda)$, We can define the set of generalized eigenvalues (resonance poles) $\hat{\sigma}_p(\mathcal{A})$ by a collection of $\lambda\in \mathbb{C}$ such that there exists some $\psi\in \kappa X$ satisfying 
$$
(\id-e^{-\lambda\tau} K^{\times}A(\lambda))\psi=0.
$$
Here $K^\times:X^\prime\to X^\prime$ is a dual operator. Similarly, we can define the generalized retarded resolvent $\mathcal{R}_\tau(\lambda):\kappa X\to X^\prime$ by 
$$
\mathcal{R}_\tau(\lambda)=A(\lambda)\circ(\id-e^{-\lambda\tau}K^\times A(\lambda))^{-1}.
$$

\noindent\textbf{Compactness assumption.}
In the representation of $\mathcal{R_\tau}(\lambda)$, $K^\times A(\lambda)$ is well defined and maps $\kappa X$ into itself due to some duality assumptions (see \textbf{(H5-6)} in Section~\ref{sec:3}). The compactness assumption refers to uniform compactness of $\kappa ^{-1}K^\times A(\lambda)\kappa:X\to X$ with respect to $\lambda$. That is, there exists some neighborhood $V$ such that $\kappa ^{-1}K^\times A(\lambda)\kappa(V)$ is relatively compact in $X$ for all $\lambda$. 

In fact, it will be shown that the linear operator $K^\times A(\lambda)$ provides a analytic continuation of $K(\lambda-iH)^{-1}$. The compactness assumption stated here corresponds to some relatively compactness of the perturbation $K$ and is satisfied in a lot of practical problems like the Friedrichs models \cite{gp} and evolution equations of Schrödinger type.  
\subsection{Main results}
Denoted by $\hat{\sigma}(\mathcal{A})$, the generalized spectrum associated with the delay equation  will be defined by the collection of singularities of $\mathcal{R}_\tau(\lambda)\kappa:X\to X^\prime$. The generalized spectrum is shown to have properties similar to those of the classical spectral theory, even though $X$ is not necessarily normable. 
Most importantly, the Riesz-Schauder theorem can be applied to the Gelfand tripe to show the following (Theorem \ref{thm:riesz schauder} in Section~\ref{sec:3}).
\begin{theorem}
    Assume that $\kappa^{-1}K^\times A(\lambda)\kappa:X\to X$ is compact, uniformly in $\lambda$. We have $\hat{\sigma}(\mathcal{A})=\hat{\sigma}_p(\mathcal{A})$.
\end{theorem}
It implies that under the compactness condition as stated before, only isolated generalized eigenvalues exist on the Riemann surface of $\mathcal{R}_\tau(\lambda)\kappa$ (i.e. $\hat{\sigma}_c(\mathcal{A})=\hat{\sigma_r}(\mathcal{A})=\emptyset$).
Specifically, we know that for any $\psi,\varphi\in X$, $(\Delta(\lambda)^{-1}\psi,\varphi)$ also has an analytic continuation $\langle \mathcal{R}_\tau(\lambda)\kappa(\psi)|\varphi\rangle$ across the imaginary axis to the left half-plane. 

Hence, contour deformation can be applied to the inverse Laplace representation of weak solutions to the delay equation (\ref{eq:sec1_linear evolution}). We can obtain a spectral expansion (see Corollary \ref{coro:eigenexpansion} in Section~\ref{sec:3}) in which the generalized eigenvalues (resonance poles) dominate the asymptotics of solutions. 
Consequently, we can still obtain an (exponential) asymptotic behavior of weak solutions in the case continuous spectrum (in the classical sense) intersects the imaginary axis. 

The theory is applied to show linear stability of the incoherence of the infinite dimensional Kuramoto model with time delay.

The Kuramoto model, known as a relatively simple mean-field model of coupled phase oscillators, plays an important role in studying the collective synchronization phenomenon \cite{kuramoto,strogatza}. By incorporating delayed mean-field coupling, the Kuramoto model is modified as \cite{yeung}
\begin{equation}{\label{eq:sec1_km1}}
    \frac{\di}{\di t}\theta_i(t)=\omega_i+\frac{k}{N}\sum_{j=1}^{N}\sin{(\theta_j(t-\tau)-\theta_i(t))},
\end{equation}
for $i=1,2,...,N.$
Here $\theta_i(t)$ denotes the phase of the $i$-th oscillator on $\mathbb{T}^1:=\mathbb{R}/2\pi \mathbb{Z}$ at time $t$. $\omega_i$ denotes the natural frequency, sampled from a probability density function $g(\omega)$ with mean $\omega_0$. $k$ is called the coupling strength, varying from the whole real axis. 

The synchronization transition can be identified using the complex order parameter, which reads
\begin{equation}
    r(t):=\frac{1}{N}\sum_{i=1}^{N}e^{ i\theta_i(t)}.
\end{equation}
It is known that if the coupling strength is smaller than some threshold, $|r(t)|$ associated with each trajectory of (\ref{eq:sec1_km1}) tends to $0$ as $t\to \infty$. On the other hand, the synchronization phenomenon appears in the strong coupling regime and each synchronization state can be identified as the quantity $|r(t)|$ tends to some positive value. 

When $N\to \infty$, dynamics of the Kuramoto model can be studied using the continuity equation (the Fokker-Planck equation with zero diffusion). The Continuum limit of (\ref{eq:sec1_km1}) reads 
\begin{equation}\label{eq:sec1_km3}
    \begin{cases}
    \partial _t \rho+\partial_\theta(V[\rho]\rho)=0,\\
     V[\rho]=\omega+k\int_{-\infty}^{+\infty}\int_0^{2\pi}\sin(\theta_\ast-\theta)\rho(t-\tau,\omega_\ast,\theta_\ast)g(\omega_\ast)\di \theta_\ast \di \omega_\ast.\\
    \end{cases}
\end{equation}
Here, $\rho(t,\omega,\theta)$ is the density function of oscillators currently at phase $\theta$ with frequency $\omega$. The normalization condition is $\int_0^{2\pi}\rho(t,\omega,\theta)=1$ for all $\omega$ and $t$. The order parameter is modified as follows 
\begin{equation}
    r(t)=|r(t)|e^{ i\psi(t)}=\int_{-\infty}^{+\infty}\int_0^{2\pi}e^{ i\theta}\rho(t,\omega,\theta)g(\omega)\di \theta \di \omega.
\end{equation}

    A trivial steady state of (\ref{eq:sec1_km3})
    \begin{equation}
        \rho(t,\omega,\theta)\equiv\frac{1}{2\pi}
    \end{equation}
    is called the incoherent state. It describes the completely de-synchronous state where all the oscillators are uniformly distributed on $\mathbb{T}^1$ (i.e. $r(t)\equiv0$).

According to linearization (see Section~\ref{sec:4}), the task reduces to establishing the asymptotic stability of the steady state for the following linear partial functional differential equation. 
\begin{equation}\label{eq:sec1_kuramoto}
        \partial_t u(t,\omega)= i\omega u(t,\omega)+\frac{k}{2}\int_{-\infty}^{+\infty}u(t-\tau,\omega)g(\omega)\di \omega.
\end{equation}
Here $k\in \mathbb{R}$ serves as a bifurcation parameter and $g(\omega)$ is assumed to be Gaussian
$
g(\omega)=\sqrt{\frac{h}{\pi}}e^{-h(\omega-\omega_0)^2}.
$
$u(t,\cdot)\in L^2(\mathbb{R},g(\omega)\di \omega)$. Denote by $\mathbb{I}$, a constant function $\mathbb{I}\equiv1$. The order parameter $r(t)=(u(t),\mathbb{I})_{L^2(\mathbb{R},g(\omega))}$.

We prove that there exists a Gelfand triple $\text{Exp}\subset L^2(\mathbb{R},g(\omega)\di \omega)\subset \text{Exp}^\prime$ and a critical value $k_c(\tau)=\frac{2}{\tau}\arccos(\cos(\omega_0\tau))-\frac{\pi}{\tau}$ such that the following holds (Theorem \ref{thm:linear_stability} in Section \ref{sec:4}). 
\begin{theorem}
    When $|k|<|k_c(\tau)|$ and $\omega_0\not=0$, there exists $h^\ast$ such that for any $h>h^\ast$ and initial data $f\in C([-\tau,0];\mathrm{Exp})$ with $f(0)=x$, $r(t)$ decays exponentially to zero as $t\to\infty$. 
\end{theorem}

The stability region ($|k|<|k_c(\tau)|$) is illustrated in Fig.~\ref{fig:2}. We point out that in the literature of Kuramoto model with time delay \cite{yeung,wd}, only instability of the incoherence can be stated. For further details on the Landau damping phenomenon in Kuramoto model we refer the reader to, for example, \cite{smm,chibaA,fernandez,medvedev,ha}.
\begin{figure}
    \centering
    \includegraphics[width=0.5\linewidth]{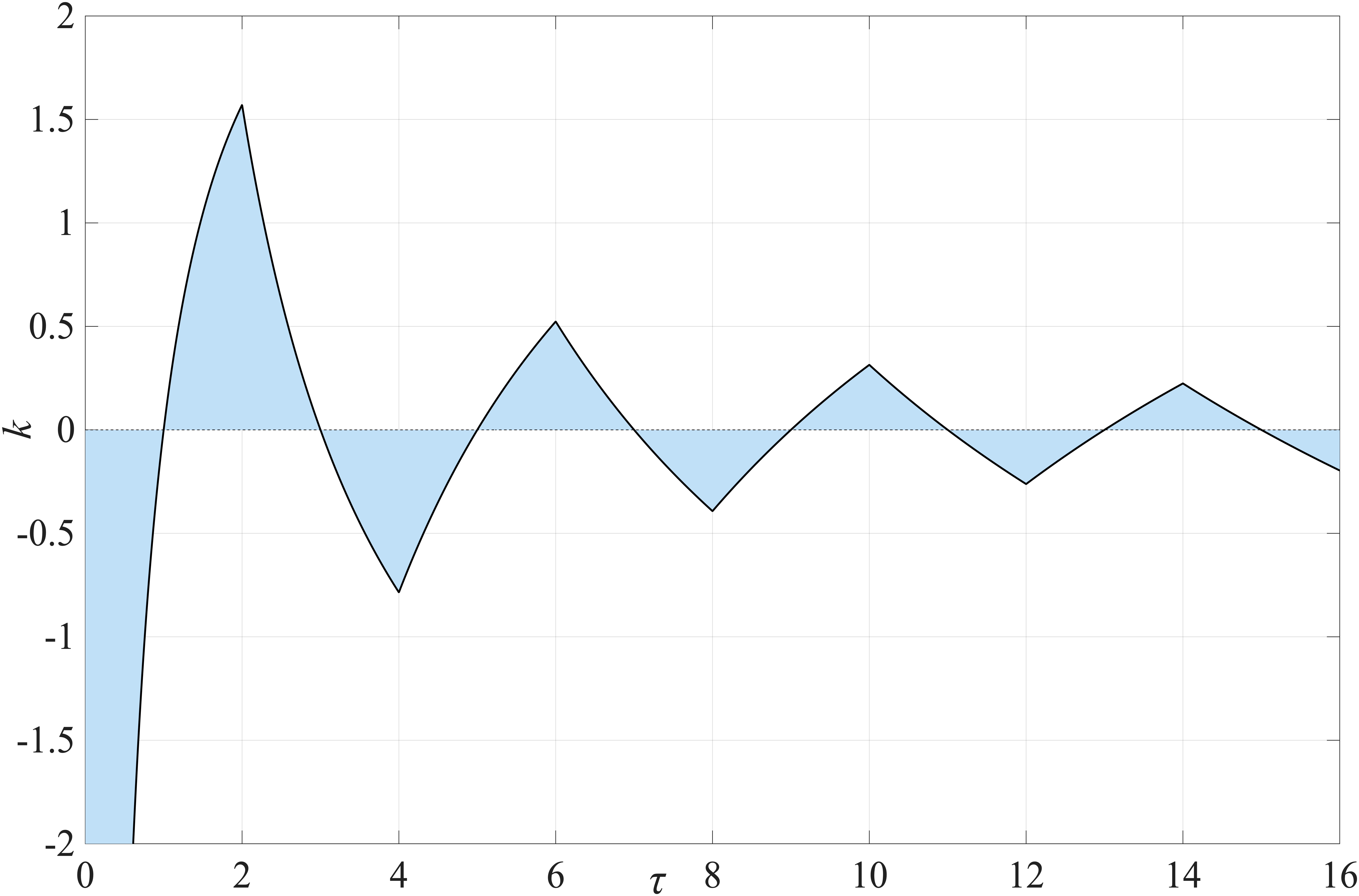}
    \caption{Linear stability region ($|k|<|k_c(\tau)|$) for the incoherent state with $\omega_0=\pi/2$.}
    \label{fig:2}
\end{figure}
\subsection{Outline}
The remainder of this paper is organized as follows.

In Section \ref{sec:2}, a classical spectral theory of delay equation in Hilbert space is introduced, following \cite{batkai1,batkai2}. 

In Section \ref{sec:3}, a generalized spectral theory for delay equation is developed via a Gelfand triple. Without relying on Evans functions or scattering matrices, Resonance poles as introduced in Definition \ref{def:resonance_pole} are defined by generalizing the characteristic equation $\Delta(\lambda)x=0$. Later, by generalizing the retarded resolvent $\Delta(\lambda)^{-1}:\mathcal{H}\to \mathcal{H}$ to the generalized resolvent $\mathcal{R}_\tau(\lambda)\kappa:X\to X^\prime$, we define the generalized spectrum as a collection of singularities of the latter (see Definition \ref{def:generalized_spectrum}).
We show that the generalized resolvent provides an analytic continuation of $\Delta(\lambda)^{-1}$ (see Proposition \ref{pro:analytical_resolvent} and \ref{pro:generalized_resolvent_continuation}) and several spectral properties of the generalized retarded resolvent are introduced. One of the main theorems is stated in Theorem \ref{thm:riesz schauder} revealing the "disappearance" of continuous spectrum under compactness condition. The asymptotics of solutions will be addressed by studying the inverse Laplace formula (see Corollary \ref{coro:eigenexpansion}).

In Section \ref{sec:4}, as an application, we prove the linear stability of the incoherence of the delayed Kuramoto model in the weak coupling regime. The conclusion is stated in Theorem \ref{thm:linear_stability}.
\section{Spectral theory on Hilbert space}\label{sec:2}
Throughout this section, we investigate the following delayed linear evolution equation
\begin{equation}\label{eq:delay_equation}
    \frac{\di}{\di t}u(t)=Au(t)+\Phi u_t,
\end{equation}
in a Hilbert space $\mathcal{H}$. Here, $A$ is a densely defined closed operator. We assume that $A$ generates a $C_0$ semigroup on $\mathcal{H}$. Called history segment, $u_t:[-\tau,0]\to \mathcal{H}$ is defined by $u_t(s):=u(t+s)$. The delay operator $\Phi$ is a bounded operator from $H^1([-\tau,0];\mathcal{H})$ to $\mathcal{H}$, where $H^1([-\tau,0];\mathcal{H}):=W^{1,2}([-\tau,0];\mathcal{H})$ is a Sobolev space defined in the usual sense. 
\begin{definition}\label{def:classical_solution}
    A solution $u:[-\tau,\infty)\to \mathcal{H}$ of (\ref{eq:delay_equation}) is called a classical solution if \\
    \noindent $\textbf{(a)}$ $u(t)\in C([-\tau,\infty];\mathcal{H})\cap C^1([0,\infty);\mathcal{H})$,\\
    \noindent $\textbf{(b)}$ $u(t)$ satisfies (\ref{eq:delay_equation}) for all $t\geq 0$,\\
    \noindent $\textbf{(c)}$ $u(t)\in \mathrm{Dom}(A)$ and $u_t(s)\in H^1([-\tau,0];\mathcal{H})$ for all $t\geq 0$. 
\end{definition}
If there exists a function $\eta:[-\tau,0]\to \mathcal{L}(\mathcal{H})$ of bounded variation such that $\Phi$ is defined by Riemann-Stieltjes integral
\begin{equation}\label{eq:phi}
    \Phi f=\int_{-\tau}^{0}\di \eta(s)f(s),
\end{equation}
for all $f\in C([-\tau,0];\mathcal{H})$, we can obtain the well-posedness of the abstract delay equation (\ref{eq:delay_equation}) (for details, see \cite{batkai1}). More precisely, we could obtain a one-to-one correspondence between solutions of (\ref{eq:delay_equation}) and solutions of a linear system on the product Hilbert space $\mathcal{S}=\mathcal{H}\times L^2([-\tau,0];\mathcal{H})$:
\begin{equation}\label{eq:delay_equation2}
    \frac{\di}{\di t}\mathcal{U}(t)=\mathcal{A}\mathcal{U}(t):=
    \begin{pmatrix}
        A & \Phi\\
        0 & \frac{\di}{\di s}
    \end{pmatrix}
    \mathcal{U}(t),
\end{equation}
where the linear operator $\mathcal{A}$ has a domain given by
$$
\text{Dom}(\mathcal{A})=\{(x,f)\in \text{Dom}(A)\times H^1([-\tau,0];\mathcal{H}) |f(0)=x\}.
$$
The correspondence theorem of (\ref{eq:delay_equation}) and (\ref{eq:delay_equation2}) is as following.
\begin{lemma}[\cite{batkai1}]\label{lemma:correspondence}
    Let $(x,f)\in \mathrm{Dom}(\mathcal{A})$ and $u(t)$ be a classical solution of (\ref{eq:delay_equation}). Then, $(u(t),u_t)$ is a classical solution of (\ref{eq:delay_equation2}). Conversely, let $\mathcal{U}(t)=(v(t),\omega(t))$ be a classical solution of (\ref{eq:delay_equation2}). Then, the function defined by
    \begin{equation}
        u(t)=\begin{cases}
            v(t),\ 0\leq t,\\
            f(t),\ -\tau\leq t<0,
        \end{cases}
    \end{equation}
    is a classical solution of (\ref{eq:delay_equation}).
\end{lemma}
When (\ref{eq:phi}) is satisfied, $\mathcal{A}$ is an infinitesimal generator that generates a $C_0$ semigroup $\{\mathcal{T}(t)\}_{t\geq0}$ on $\mathcal{S}$. It follows from Lemma \ref{lemma:correspondence} that the classical solution of the delay equation can be written as 
\begin{equation}\label{eq:projection}
    u(t)=\begin{cases}
        \pi_1\circ\mathcal{T}(t)(x,f),\ 0\leq t,\\
        f(t),\ -\tau\leq t<0,
    \end{cases}
\end{equation}
where $\pi_1:\mathcal{S}\to \mathcal{H}$ denotes a projection map and $(x,f)\in \text{Dom}(\mathcal{A})$.  

It is easy to verify that (\ref{eq:projection}) is well defined when $(x,f)$ take values in the whole state space $\mathcal{S}$. In this case, (\ref{eq:projection}) is called a mild solution of the delay equation. 
\subsection{Spectral properties of the infinitesimal generator $\mathcal{A}$}
It is known that the spectrum of a linear operator can be regarded as a set of singularities of its corresponding resolvent. To obtain the spectral properties of $\mathcal{A}$, we calculate the following linear equations:
\begin{equation}\label{eq:delay_resolvent_calculate}
\begin{pmatrix}
    \lambda-A & -\Phi\\
    0 & \lambda-\frac{\di}{\di s}
\end{pmatrix}
\begin{pmatrix}
    x\\
    f
\end{pmatrix}
    =\begin{pmatrix}
        y\\
        g
    \end{pmatrix},
\end{equation}
where $\Phi=\int_{-\tau}^0\di \eta(s)$.
We obtain
$$
f =e^{\lambda s}x-\int_0^se^{\lambda(s-m)}g(m)\di m=e^{\lambda s}x+(\lambda-A_L)^{-1}g,
$$
where $x$ satisfies
$$
(\lambda -A-\int_{-\tau}^0e^{\lambda s}\di \eta(s))x=y+\int_{-\tau}^0\di \eta(s)(\lambda-A_L)^{-1}[g](s). 
$$
$A_L$ denotes the infinitesimal generator of the nilpotent left shift semigroup on $L^2([-\tau,0];\mathcal{H})$, $L(t)$. More precisely, $L(t)$ is defined by
$$
    L(t)[f](s)=\begin{cases}
        0,\ -t<s\leq 0,\\
        f(t+s),\ -\tau\leq s\leq -t.
    \end{cases}
$$
\begin{remark}
    It is easy to verify that 
    \begin{equation}
        A_L [f](s)=\frac{\di}{\di s}f(s),
    \end{equation}
    and $\text{Dom}(A_L)=\{f\in H^1([-\tau,0];\mathcal{H})|f(0)=0)\}$. The resolvent operator of $A_L$ can be obtained as 
    \begin{equation}
        (\lambda-A_L)^{-1}[f](s)=-\int_0^s e^{\lambda(s-m)}f(m)\di m.
    \end{equation}
    The resolvent exists for any $\lambda\in \mathbb{C}$ and it is easy to verify that the spectrum of $A_L$ is empty set ($\sigma(A_L)=\emptyset$).
    
\end{remark}
Denote by $\varepsilon_\lambda$ an exponential function of $s\in [-\tau,0]$: $\varepsilon_\lambda:=e^{\lambda s}$. 
For simplicity, Let $\Delta(\lambda)$ be a linear operator on $\mathcal{H}$ defined by
\begin{equation}
    \Delta(\lambda)x:=(\lambda -A-\Phi(\varepsilon _\lambda\otimes \id))x=(\lambda-A)x-\int_{-\tau}^0e^{\lambda s}\di \eta(s)x. 
\end{equation}
Here $\Phi(\varepsilon_\lambda\otimes \id):\mathcal{H}\to \mathcal{H}$ is a bounded operator. Moreover, from the properties of $L(t)$ and its generator $A_L$, it is clear that $\Phi(\varepsilon_\lambda\otimes \id)$ is an operator-valued holomorphic function with respect to $\lambda$. We can obtain the following lemma. 
\begin{lemma}[\cite{batkai1}]\label{lemma:delay_generalized_retarded}
    For $\lambda\in \mathbb{C}$, we have $\lambda\in \varrho(\mathcal{A})$ if and only if $\lambda\in \varrho(A+\Phi(\varepsilon_\lambda \otimes \id))$. Moreover, for $\lambda\in \varrho(\mathcal{A})$, the resolvent $(\lambda-\mathcal{A})^{-1}$ is given by
    \begin{equation}\label{eq:resolvent_A}
    (\lambda-\mathcal{A})^{-1}=
        \begin{pmatrix}
            \Delta(\lambda)^{-1}& \Delta(\lambda)^{-1}\Phi (\lambda-A_L)^{-1}\\
            \varepsilon_\lambda\otimes \Delta(\lambda)^{-1} & [\varepsilon_\lambda\otimes \Delta(\lambda)^{-1}+\id](\lambda-A_L)^{-1}
        \end{pmatrix}.
    \end{equation}
\end{lemma}
From Lemma \ref{lemma:delay_generalized_retarded}, we know that even though the linear operator $\mathcal{A}$ defines a flow on a product Hilbert space $\mathcal{S}$, the spectrum of $\mathcal{A}$ can be solely determined by singularities of a linear operator $\Delta(\lambda)^{-1}$ on $\mathcal{H}$. It is convenient to call $\Delta(\lambda)^{-1}:\mathcal{H}\to \mathcal{H}$ the retarded resolvent operator \cite{nakagiri2}.  

Specifically, we assume $(y,g)=0$ in (\ref{eq:delay_resolvent_calculate}). Then, we know that $\lambda\in \mathbb{C}$ is an eigenvalue of $\mathcal{A}$ if there exists some $x\in \mathcal{H}$ that solves the characteristic equation
\begin{equation}
    \Delta(\lambda)x=0.
\end{equation}
To further illustrate a classification of spectrum of $\mathcal{A}$, it is worth noticing a decomposition of (\ref{eq:resolvent_A}):
$$
(\lambda-\mathcal{A})^{-1}=LMR=\begin{pmatrix}
    \id & 0\\
    \varepsilon_\lambda\otimes \id & \id
\end{pmatrix}
\begin{pmatrix}
    \Delta(\lambda)^{-1} & 0\\
    0 & (\lambda-A_L)^{-1}
\end{pmatrix}
\begin{pmatrix}
    \id & \Phi (\lambda-A_L)^{-1}\\
    0 & \id
\end{pmatrix}.
$$
Here, $L$ and $M$ are bounded and invertible on $\mathcal{S}$. Combined with this fact and the boundedness of $(\lambda-A_L)^{-1}$, we can prove that $\lambda\in \mathbb{C}$ belongs to the point/continuous/residual spectrum of $\mathcal{A}$ if and only if $\lambda$ belongs to the the point/continuous/residual spectrum of $A+\Phi(\varepsilon_\lambda\otimes \id)$. 
\begin{lemma}
    We have the following.\\
    \noindent $\mathbf{(a)}$ $\lambda\in \sigma_p(\mathcal{A})$ if and only if $\Delta(\lambda)$ is not injective,\\
    \noindent $\mathbf{(b)}$ $\lambda\in \sigma_c(\mathcal{A})$ if and only if $\mathrm{Dom}(\Delta(\lambda)^{-1})$ is a proper dense subspace of $\mathcal{H}$,\\
    \noindent $\mathbf{(c)}$ $\lambda\in \sigma_r(\mathcal{A})$ if and only if $\mathrm{Dom}(\Delta(\lambda)^{-1})$ is not dense,\\
    \noindent $\mathbf{(d)}$ $\lambda\in \sigma_{ess}(\mathcal{A})$ if and only if $\Delta(\lambda)$ is not Fredholm. 
\end{lemma}
It is known that solution of the Cauchy problem of (\ref{eq:delay_equation2}) can be obtained by applying inverse Laplace transform to $(\lambda-\mathcal{A})^{-1}$. The solution of (\ref{eq:delay_equation}) which initializes at $f\in H^1([-\tau,0];\mathcal{H})$ with $x=f(0)$ can be represented by
$$
u(t)=\pi_1\circ\frac{1}{2\pi i}\int_\Gamma e^{\lambda t}(\lambda-\mathcal{A})^{-1}(x,f)^\top\di \lambda,
$$
where the integral contour $\Gamma$ is on the right of $\sigma(\mathcal{A})$. Equivalently, we have the following inverse Laplace transform formula of $u(t)$
\begin{equation}
    u(t)=\frac{1}{2\pi i}\int_\Gamma e^{\lambda t}\Delta(\lambda)^{-1}x\di \lambda+\frac{1}{2\pi i}\int_\Gamma e^{\lambda t}\Delta(\lambda)^{-1}\Phi (\lambda-A_L)^{-1}f\di \lambda.
\end{equation}
The following section will focus on the retarded resolvent and the asymptotic behavior of $u(t)$ as $t\to \infty$. 
\section{Generalized spectral theory on Gelfand triple}\label{sec:3}
Throughout this section, we construct a generalized spectral theory on the following linear evolution equation with a single time lag
\begin{equation}\label{eq:aFDE}
    \frac{\di}{\di t}u(t)= iHu(t)+Ku(t-\tau).
\end{equation}
Here, $H$ is a self-adjoint operator defined on a Hilbert space $\mathcal{H}$ and $K:\mathcal{H}\to \mathcal{H}$ is bounded.

It will be shown that, under some analyticity condition of the spectral measure of $(H, \mathrm{dom}(H))$, the retarded resolvent operator $\Delta(\lambda)^{-1}:\mathcal{H}\to \mathcal{H}$ as introduced in the previous section can be generalized through a Gelfand triple.

It is verified that the continuous spectrum on the imaginary axis can "disappear", leaving only discrete generalized eigenvalues on a new Riemann surface induced by the Gelfand triple. Thanks to this discovery, the obstacles in analysis of asymptotics of solutions caused by those continuous singularities can be overcome. 
\subsection{Gelfand triples}
Let $X$ be a Hausdorff locally convex topological vector space over $\mathbb{C}$. We denote by $X^\prime$ a collection of all continuous anti-linear functionals on $X$. On the (anti-)dual space $X^\prime$, a strong topology (strong $^\star$/dual topology) and a weak topology (weak $^\star$/dual topology) can be induced by \cite{Treves,SW}:
\begin{equation}
\begin{aligned}
    &u_i \stackrel{\mathrm{weak}}{\to} u \stackrel{def}{\Longleftrightarrow} \langle u_i|\varphi\rangle\to \langle u|\varphi\rangle, \quad \forall \varphi\in X, \\
    &u_i \stackrel{\mathrm{strong}}{\to} u \stackrel{def}{\Longleftrightarrow}\langle u_i|\varphi\rangle\to \langle u|\varphi\rangle, \quad \text{uniformly on each bounded subset of }X.
\end{aligned}
\end{equation}
Here, the bracket $\langle u|\cdot\rangle:X\to \mathbb{C}$ for each $u\in X^\prime$ refers to an anti-linear map associated with the dual pair $(X^\prime,X)$.

It is clear that according to Riesz's representation theorem, each Hilbert space $\mathcal{H}$ is isomorphic to its dual space. Hence, for any dense subspace $X\subset\mathcal{H}$ endowed with a stronger topology, we can define the corresponding Gelfand triple based on the fact $\mathcal{H}\cong\mathcal{H}^\prime \subset X^\prime$.
\begin{definition}\label{def:Gelfand triple}
    Let $(\mathcal{H},(\cdot,\cdot))$ be a complex Hilbert space with inner product $(\cdot,\cdot)$. If a dense subset $X\subset \mathcal{H}$ is a Hausdorff locally convex topological vector space and its topology is stronger than $\mathcal{H}$, the triple 
    \begin{equation}
        X\subset \mathcal{H}\subset{X^\prime}
    \end{equation}
    is called a Gelfand triple or a rigged Hilbert space associated with the Hilbert space $\mathcal{H}$. Denote by $\kappa:X\to X^\prime$ the canonical inclusion is defined as
    \begin{equation}
        \langle \kappa(u)|\varphi\rangle=(u,\varphi),\quad \forall \varphi\in X. 
    \end{equation}
    It is easy to verify that the inclusion map is continuous with respect to both strong and weak dual topology.
\end{definition}
\subsection{Generalization of the characteristic equation}

The spectral theory of self-adjoint operators directly implies that the spectrum of $H$ lies on the real axis ($\sigma(H)\subset \mathbb{R}$) and $H$ has a spectral decomposition \cite{davies}\\
$$H=\int_{-\infty}^{+\infty}\omega \di E(\omega).$$
Here, $E(\cdot)$ is an operator-valued measure function called the spectral measure.

We fix the following notations: $I\subset\mathbb{R}$ is a connected interval such that $\sigma(H)\subset I$; 
$\mathbb{C_+}$ ($\mathbb{C_-}$, resp.):=$\{\lambda\in \mathbb{C}|\Im(\lambda)> 0 \ (\Im(\lambda)< 0, \ \text{resp.})\}$; $\Omega:= \mathbb{C_-}\cup I\cup \mathbb{C_+}$.

Like \cite{chibaB}, the following standing hypotheses are required to construct a generalized spectral theory:

\noindent \textbf{(H1)} $X$ is a dense subspace of $\mathcal{H}$ and it has a stronger topology.\\
\noindent \textbf{(H2)} $X$ is a quasi-complete barreled space.\\ 
\noindent \textbf{(H3)} The spectral measure of $H$ is absolutely continuous on $I$ and its density function\\ 
\begin{equation}
    E[\psi,\varphi](\omega):=\frac{\di (E(\omega)\psi,\varphi)}{\di \omega}
\end{equation} has an analytic continuation across $I$ to the upper complex half-plane $\mathbb{C}_+$.\\
\noindent \textbf{(H4)} For each $\lambda\in \mathbb{C_+}\cup I$, the bilinear form $E[\cdot,\cdot](\lambda): X\times X\to \mathbb{C}$ is separately continuous.

\begin{remark}
    \textbf{(H1)} here ensures that $X\subset \mathcal{H}\subset X^\prime$ forms a Gelfand triple in the sense of Definition \ref{def:Gelfand triple}. From \textbf{(H2)}, $X^\prime$-valued integrals are well defined in a weak sense (known as weak$^\ast$ Pettis integrals or Gelfand integrals). Moreover, Laurent series expansion and residue theorem are applicable for $X^\prime$-valued functions/integrals \cite[Appendix]{chibaB}. From absolute continuity of the spectral measure, we know that $\lambda-H$ is injective for each $\lambda\in \mathbb{R}$. 
\end{remark}

The Stone's theorem implies that $ iH$ generates a $C_0$ group on $\mathcal{H}$. As shown in Section~\ref{sec:2}, there is a one-to-one correspondence between solutions of the delay equation (\ref{eq:aFDE}) and solutions of a linear system defined on a state space $\mathcal{S}=\mathcal{H}\times L^2([-\tau,0];\mathcal{H})$:
\begin{equation}\label{eq:acp}
    \frac{\di}{\di t}
    \begin{pmatrix}
        x\\
        f
    \end{pmatrix}
    =\mathcal{A}\begin{pmatrix}
        x\\
        f
    \end{pmatrix}=
    \begin{pmatrix}
         iH & \hat{K}\\
        0 & \frac{\di}{\di s}
    \end{pmatrix}
    \begin{pmatrix}
        x\\
        f
    \end{pmatrix},
\end{equation}
with $\text{Dom}(\mathcal{A})=\{(x,f)\in \text{Dom}(H)\times H^1([-\tau,0];\mathcal{H})|f(0)=x\}$. $\hat{K}f=\int_{-\tau}^{0}\di\eta(s)f(s)$ and $\eta:[-\tau,0]\to \mathcal{L}(\mathcal{H})$ is given by
\begin{equation}
\eta(s) =
\begin{cases}
    0, \quad \text{others},\\
    K,\quad s=-\tau.
\end{cases}
\end{equation}
Hence, $\lambda$ is an eigenvalue of $\mathcal{A}$ if and only if there exists some nonzero $x\in \text{Dom}(H)$ which solves the characteristic equation
\begin{equation}\label{eq:delta}
    \Delta(\lambda)x=(\lambda- iH-e^{-\lambda\tau}K)x=0.
\end{equation}
From \textbf{(H3)}, it is clear that the above equation is equivalent to 
\begin{equation}\label{eq:eigeneq_H}
    (\id-e^{-\lambda\tau}(\lambda- iH)^{-1}K)x=0.
\end{equation}
Firstly, We are in a position to generalize the resolvent operator of $ iH:\mathcal{H}\to \mathcal{H}$ using the Gelfand triple $X\subset \mathcal{H}\subset X^\prime$. When $\lambda$ takes values in the right complex half-plane, we can define $(\lambda- iH)^{-1}:\kappa X\to X^\prime$ by 
\begin{equation}\label{eq:cauchy_integral}
    \langle(\lambda- iH)^{-1}\kappa(\psi)|\varphi\rangle:=((\lambda- iH)^{-1}\psi,\varphi)=\int_{-\infty}^{+\infty}\frac{1}{\lambda- i\omega}E[\psi,\varphi](\omega)\di \omega,
\end{equation}
for each $\psi,\varphi\in X$. 

When the density function $E[\psi,\varphi](\omega)$ has an analytic continuation to $I\cup\mathbb{C}_+$, we can define $A(\lambda):\kappa X\to X^\prime$ by
\begin{equation}\label{eq:Alambda}
    \langle A(\lambda)\kappa(\psi)|\varphi\rangle:=
    \begin{cases}
    \int_{-\infty}^{+\infty}\frac{1}{\lambda- i\omega}E[\psi,\varphi](\omega)\di \omega+2\pi E[\psi,\varphi](\frac{\lambda}{ i}), \quad \Re(\lambda)<0,\\
    \lim_{\Re(\lambda)\to0+}\int_{-\infty}^{+\infty}\frac{1}{\lambda- i\omega}E[\psi,\varphi](\omega)\di \omega, \quad \lambda\in iI,\\
        \int_{-\infty}^{+\infty}\frac{1}{\lambda- i\omega}E[\psi,\varphi](\omega)\di \omega, \quad \Re(\lambda)>0,\\
    \end{cases}
\end{equation}
for each $\psi,\varphi\in X$. It is easy to verify that $\langle A(\lambda)i(\psi)|\varphi\rangle$ is holomorphic and extends $((\lambda- iH)^{-1}\psi,\varphi)$ to the left complex half-plane. 
\begin{remark}
    Indeed, by applying change of variable $\zeta=\frac{\lambda}{ i}$, we rewrite (\ref{eq:cauchy_integral}) as 
    $$
        \langle(\lambda- iH)^{-1}\kappa(\psi)|\varphi\rangle=- iI(\zeta):=- i\int_{-\infty}^{+\infty}\frac{1}{\zeta-\omega}E[\psi,\varphi](\omega)\di \omega.
    $$
    The Cauchy-type integral $I(\zeta)$ is holomorphic on both $\mathbb{C}_+$ and $\mathbb{C}_-$. Moreover, we have
    $$
    \begin{aligned}
    \lim_{\Im(\zeta)\to0\mp}I(\zeta)=\mathrm{p.v.}\int_{-\infty}^{+\infty}\frac{1}{\zeta-\omega}E[\psi,\varphi](\omega)\di \omega\mp i\pi E[\psi,\varphi](\zeta).\\
    \end{aligned}
    $$
    Therefore, when $(\lambda- iH)^{-1}$ is extended to the left half-plane, there is a jump $2\pi E[\psi,\varphi](\frac{\lambda}{ i})$ appeared on the original formulation. Henceforth, we will call $A(\lambda)$ the analytic continuation of $(\lambda- iH)^{-1}$.
\end{remark}
\begin{lemma}
    $A(\lambda)\kappa:X\to X^\prime$ is continuous when $X^\prime$ is endowed with the weak dual topology.
\end{lemma}
\begin{proof}
    When $\Re(\lambda)>0$, we have $\langle A(\lambda)\kappa(\psi)|\varphi\rangle=((\lambda-iH)^{-1}\psi,\varphi)$. The statement directly follows from the continuity of $(\lambda-iH)^{-1}:\mathcal{H}\to \mathcal{H}$ on the right half plane. When $\lambda\in iI$, from the holomorphy of $A(\lambda)$ it follows that
    $$
    \langle A(\lambda)\kappa(\psi)|\varphi\rangle=\lim_{k\to \infty}\langle A(\lambda_k)\kappa(\psi)|\varphi\rangle,
    $$
    for some sequence $\{\lambda_k\}_{k\in \mathbb{N}}$ on right half-plane. Continuity of $A(\lambda)\kappa$ holds from the Banach-Steinhaus theorem on barreled space. From \textbf{(H4)}, for each $\varphi\in X$ and $\lambda\in \mathbb{C}_+$, $E[\cdot,\varphi](\lambda)$ is continuous, which completes the proof. 
\end{proof}
To extend the definition of the characteristic equation (\ref{eq:eigeneq_H}), we have to assume the following duality conditions. We fix the following notations. Let $\Phi$ be a densely defined linear operator on $X$. Denote by $\Phi^\prime:X^\prime\to X^\prime$ the dual operator of $\Phi$, which satisfies 
\begin{equation}
    \langle \Phi^\prime f|g\rangle = \langle f|\Phi g\rangle,\quad \forall f\in \text{Dom}(\Phi^\prime),\ g\in \text{Dom}(\Phi).
\end{equation}
Let $\Psi$ be an operator on a Hilbert space $\mathcal{H}$. Denote by $\Psi^\ast$, the dual operator of $\Psi$ (in the Hilbert sense) is defined as 
\begin{equation}
    (\Psi^\ast f,g)=(f,\Psi g), \quad \forall f\in \text{Dom}(\Psi^\ast),\ g\in \text{Dom}(\Psi).
\end{equation}
For simplicity, we denote $\Psi^{\times}$ as the bidual operator of $\Psi$ (i.e. $\Psi^{\times}:=(\Psi^{\ast})^\prime$). It is easy to show $\Psi^{\times}$ extends $\Psi$ in the sense of $\Psi^{\times}|_{\kappa\text{Dom}(\Psi)}=\kappa \Psi\kappa^{-1}$.

Now, we assume that\\
\noindent \textbf{(H5)} there exists some dense subspace $Y\subset X$ satisfying $HY\subset X$ and $K^\ast Y\subset X$.\\
\noindent \textbf{(H6)} $\kappa^{-1}K^\times A(\lambda)\kappa X\subset X$ for any $\lambda\in \Omega$.

Under \textbf{(H1-6)}, we can define the generalized eigenvalue associated with the delay equation (\ref{eq:aFDE}) by extending (\ref{eq:eigeneq_H}) as
\begin{definition}\label{def:resonance_pole}
    $\lambda\in \Omega$ is called a generalized eigenvalue (or a resonance pole) of (\ref{eq:aFDE}) if there exists nonzero $x\in \text{Ran}(A(\lambda))\subset X^\prime$ such that 
    \begin{equation}\label{eq:resonance_pole1}
    (\id-e^{-{\lambda\tau}}A(\lambda)K^{\times})x=0.
    \end{equation}
\end{definition}
It is easy to obtain an equivalent definition saying that $\lambda$ is a generalized eigenvalue if and only if we can find some $x\in \kappa X$ such that 
\begin{equation}\label{eq:resonance_pole2}
    (\id-e^{-{\lambda\tau}}K^{\times}A(\lambda))x=0.
\end{equation}
Indeed, if (\ref{eq:resonance_pole1}) is satisfied with $x\in \text{Ran}(A(\lambda))$, we can get 
\begin{equation}
    K^\times\circ(\id-e^{-\lambda\tau}A(\lambda)K^{\times})x=(\id-e^{-\lambda\tau}K^{\times}A(\lambda))\circ K^\times x=0.
\end{equation}
Conversely, if there exists some $x\in \kappa X$ satisfying (\ref{eq:resonance_pole2}), by applying $A(\lambda)$ to both sides of (\ref{eq:resonance_pole2}), we have
\begin{equation}
    A(\lambda)\circ(\id-e^{-\lambda\tau}K^{\times}A(\lambda))x=(\id-e^{-\lambda\tau}A(\lambda)K^{\times})\circ A(\lambda)x=0.
\end{equation}

From another perspective, the characteristic equation can also be generalized by taking biduality of each component of $\Delta(\lambda)$. We define 
\begin{equation}
    \Delta(\lambda)^\times=\lambda- iH^{\times}-e^{-\lambda\tau}K^{\times}:X^\prime\to X^\prime.
\end{equation}
Here $H^\times=H^\prime$ since $H$ is self-adjoint. We have the following.
\begin{proposition}\label{pro:generalized_eigenvalue}
    Let $\lambda$ be a generalized eigenvalue and $(\id-e^{-\lambda\tau}A(\lambda)K^\times)x=0$ for some $x\in \mathrm{Ran}(A(\lambda))$. Then, we have
    $\Delta(\lambda)^\times x =0.$
\end{proposition}
\begin{proof}
    Firstly, we can obtain that the following identity holds for any $\psi,\varphi\in \text{Dom}(H) $ and $\lambda\in \mathbb{C}$:
    \begin{equation}\label{eq:functional_calculus}
        E[\psi,(\bar{\lambda}-H)\varphi](\omega)=(\lambda-\omega)E[\psi,\varphi](\omega).
    \end{equation}
    Indeed, we have 
    $$
    \begin{aligned}
        E[\psi,(\bar{\lambda}-H)\varphi](\omega)=\frac{\di (E(\omega)\psi,(\bar{\lambda}-H)\varphi)}{\di \omega}&=\frac{\di }{\di \omega}([\int_{-\infty}^{+\infty}(\lambda-s)\di E(s)]\circ E(\omega)\psi,\varphi)\\
        &=\frac{\di}{\di \omega}((\lambda-\omega)E(\omega)\psi,\varphi)=(\lambda-\omega)E[\psi,\varphi](\omega).
    \end{aligned}
    $$
    According to (\ref{eq:Alambda}), when $\Re(\lambda)<0$, the analytic continuation $A(\lambda)$ is given by  
    $$
    \langle A(\lambda)\kappa(\psi)|\varphi\rangle=\int_{-\infty}^{+\infty}\frac{1}{\lambda- i\omega}E[\psi,\varphi](\omega)\di \omega+2\pi E[\psi,\varphi](\frac{\lambda}{ i}).
    $$
    Applying (\ref{eq:functional_calculus}), we can get that for each $\psi\in X$ and $\varphi\in Y$ (see \textbf{(H5)} for definition of $Y$), it follows that
    $$
    \begin{aligned}
    \langle A(\lambda) \kappa(\psi)|(\bar{\lambda}+ iH)\varphi\rangle&= \int_{-\infty}^{+\infty}E[\psi,\varphi](\omega)\di \omega+2\pi (\lambda- i\omega)E[\psi,\varphi](\omega)|_{\omega=\frac{\lambda}{ i}}   \\
    &=\int_{-\infty}^{+\infty}E[\psi,\varphi](\omega)\di \omega=(\psi,\varphi)=\langle \kappa(\psi)|\varphi\rangle.
    \end{aligned}
    $$
    Hence, we can obtain $(\lambda- iH^\times)A(\lambda)=\id$. Since $\lambda$ is a generalized eigenvalue and $(\id-e^{-\lambda\tau}A(\lambda)K^\times)x=0$. Applying $\lambda- iH^{\times}$ from the left, we can get $\Delta(\lambda)^{\times}x=0.$ Here, we only present the proof by assuming $\Re(\lambda)<0$ because others follow from similar arguments. 
\end{proof}
\begin{remark}
    We denote by $\hat{\sigma}_p(\mathcal{A})$ the set of all generalized eigenvalues. It follows from Proposition \ref{pro:generalized_eigenvalue} that we have an inclusion relation  
    $$
    \{\lambda|\text{Ker}(\Delta(\lambda))\not=\{0\}\}\subset \hat{\sigma}_p(\mathcal{A}) \subset \{\lambda|\text{Ker}(\Delta(\lambda)^\times)\not=\{0\}\}.
    $$
    In general, compared to the usage of a Gelfand triple, the dual space $X^\prime$ is too large to induce a generalized spectral theory for linear operators. Further discussion can be found in \cite{chibaB}.
\end{remark}
\subsection{Generalized retarded resolvent}
From Section~\ref{sec:2}, we know that spectrum of 
\begin{equation}
\mathcal{A}=\begin{pmatrix}
     iH & \hat{K}\\
    0 & \frac{\di}{\di s}
\end{pmatrix}:\mathcal{S}=\mathcal{H}\times L^2([-\tau,0];\mathcal{H})\to \mathcal{S}
\end{equation}
is solely determined by singularities of the so-called retarded resolvent operator 
\begin{equation}\label{eq:retarded_resolvent}
    \Delta(\lambda)^{-1}:=(\lambda- iH-e^{-\lambda\tau}K)^{-1}:\mathcal{H}\to \mathcal{H}.
\end{equation}
Equivalently, we have
\begin{equation}
    \Delta(\lambda)^{-1}=(\lambda- iH)^{-1}\circ(\id -e^{-\lambda\tau}K(\lambda- iH)^{-1})^{-1}.
\end{equation}
Like in the preceding subsection, we obtain the following definition.
\begin{definition}\label{def:generalized_retarded_resolvent}
    Let $\lambda\in \Omega$ be such that the inverse of $\id-e^{-\lambda\tau}K^\times A(\lambda)$ exists. The generalized retarded resolvent operator $\mathcal{R}_\tau(\lambda):\kappa X\to X^\prime$ is defined as 
    \begin{equation}
        \mathcal{R}_\tau(\lambda):=A(\lambda)\circ(\id-e^{-\lambda\tau}K^\times A(\lambda))^{-1}.
    \end{equation}
\end{definition}
As shown before, $\id-e^{-\lambda\tau}K^\times A(\lambda)$ is injective if and only if $\id-e^{-\lambda\tau}A(\lambda)K^\times $ is injective on $\text{Ran}(A(\lambda))\subset X^\prime$. Hence, we can obtain an equivalent expression of $\mathcal{R}_{\tau}(\lambda)$ as 
\begin{equation}
    \mathcal{R}_{\tau}(\lambda)=(\id-e^{-\lambda\tau}A(\lambda)K^\times)^{-1}\circ A(\lambda).
\end{equation}
It has been proven that when $X^\prime$ is endowed with the weak dual topology, then $A(\lambda)\kappa:X\to X^\prime$ is continuous. In this sense, it is natural to relate the concept of a generalized resolvent set with some continuity conditions of $\mathcal{R}_\tau(\lambda)\kappa:X\to X^\prime$ instead of $\mathcal{R}_\tau(\lambda):\kappa X\to X^\prime$.

However, unlike the definition of the resolvent set of linear operators in Banach spaces, the definition has to be modified when the linear operator is defined in a general Hausdorff locally convex topological vector space. More precisely, to ensure that the defined resolvent set is open, in the fashion of Waelbroeck \cite{wael}, we have the following.
\begin{definition}
    $\lambda\in \Omega$ is called an element of the generalized resolvent set if there exists a neighborhood of $\lambda$, $V_{\lambda}$ such that for each $\mu\in V_{\lambda}$, $\mathcal{R}_{\tau}(\mu)\kappa:X\to X^\prime$ is a densely defined continuous operator where $X^\prime$ is equipped with the weak dual topology and $\{\mathcal{R}_{\tau}(\mu)\kappa(\psi)\}_{\mu\in V_{\lambda}}\subset X^{\prime}$ is a bounded subset for each $\psi\in X$.
\end{definition}
Let us denote by $\hat{\varrho}(\mathcal{A})$ the generalized resolvent set and $\hat{\sigma}(\mathcal{A}):=\Omega\backslash \hat{\varrho}(\mathcal{A})$ the generalized spectrum. Despite the set of generalized eigenvalues $\hat{\sigma}_{p}(\mathcal{A})$ as defined in Definition \ref{def:resonance_pole}, we have the following finer classification of the generalized spectrum.
\begin{definition}\label{def:generalized_spectrum}
    The set of generalized residual spectrum $\hat{\sigma}_{r}(\mathcal{A})$ is a collection of all $\lambda\in \Omega$ s.t. the domain of $\mathcal{R}_{\tau}(\mu)\kappa$ is not dense in $X$. Denote by $\hat{\sigma}_c (\mathcal{A}):=\hat{\sigma}(\mathcal{A})\backslash(\hat{\sigma}_{p}(\mathcal{A})\cup \hat{\sigma}_{r}(\mathcal{A}))$ the remainder of the generalized spectrum is called generalized continuous spectrum.
\end{definition}
\begin{remark}
    When $X$ is a Banach space, the above definitions for both generalized resolvent set and generalized spectrum are equivalent to the classical definitions of resolvent/spectrum. However, when $X$ is not normable (Fréchet space, for example), we can find counterexamples to show $\hat{\sigma}(\mathcal{A})$ does not coincide with the spectrum defined in the Banach sense. We refer \cite{maeda} as a good note for this point. 
\end{remark}
Detailed analysis of properties of generalized spectrum will be presented in the next subsection as one of the main parts of this paper.

It is worth mentioning that the generalized retarded resolvent $\mathcal{R}_\tau(\lambda)$ maps $\kappa X$ into a different space. The composition $\mathcal{R_\tau}(\lambda)\circ\mathcal{R_\tau}(\mu) $ is not well defined in general. As an immediate consequence of this observation, the powerful resolvent identities become problematic in the generalized spectral theory constructed on Gelfand triples. Nevertheless, we still have holomorphy of the resolvent.
\begin{proposition}\label{pro:analytical_resolvent}
For any $\psi\in X$, $\mathcal{R}_\tau(\lambda)\kappa(\psi)$ is an $X^\prime$-valued holomorphic function in $\lambda\in \hat{\varrho}(\mathcal{A})$
\end{proposition}
\begin{proof}
    First of all, we claim that for each $\psi\in X$, $\kappa^{-1}K^{\times}A(\lambda)\kappa(\psi)$ is an $X$-valued holomorphic function. Indeed, for each $\varphi\in Y\subset X$, we have
    $$
    \begin{aligned}
        \langle \kappa(\varphi)|\kappa^{-1}K^{\times}A(\lambda)\kappa(\psi)\rangle=(\varphi,\kappa^{-1}K^{\times}A(\lambda)\kappa(\psi))&=\overline{\langle K^{\times}A(\lambda)\kappa(\psi)|\varphi\rangle}=\overline{\langle A(\lambda)\kappa(\psi)|K^\star\varphi\rangle}.
    \end{aligned}
    $$
    $Y$ is dense in $X$, and thus dense in $X^\prime $. Because $\langle A(\lambda)\kappa(\psi)|K^\star\varphi\rangle$ is holomorphic, it is clear that $\kappa^{-1}K^{\times}A(\lambda)\kappa(\psi)$ is weakly holomorphic in $\lambda$. The strong holomorphy directly follows from the quasi-completeness of $X$.

    If $\lambda\in \hat{\varrho}(\mathcal{A})$, there exists a (bounded) neighborhood $V_{\lambda}$ s.t. $\mathcal{R}_\tau(\mu)\kappa(\psi)$ exists for each $\mu\in V_\lambda$. In addition, $\{\mathcal{R}_\tau(\mu)\kappa(\psi)\}_{\mu\in V_\lambda}$ is a bounded subset of $X^\prime$. In other words, we know that $\{\mathcal{R}_\tau(\mu)\kappa\}_{\mu\in V_\lambda}$ as a family of linear operators from $X$ to $X^\prime$ is bounded for the topology of pointwise convergence. Since $X$ is a barreled space, the Banach-Steinhaus theorem implies that $\{\mathcal{R}_\tau(\mu)\kappa\}_{\mu\in V_\lambda}$ is equicontinuous \cite{Treves}. Denote by $\psi_{\lambda,\tau}:=\kappa^{-1}[\id-e^{-\lambda\tau}K^\times A(\lambda)]^{-1}\kappa(\psi)$ an element in $X$. It follows that
    $$
    \begin{aligned}
        &\mathcal{R}_\tau(\lambda+h)\kappa(\psi)-\mathcal{R}_\tau(\lambda)\kappa(\psi)\\&=A(\lambda+h)\kappa\circ \kappa^{-1}[\id -e^{-(\lambda+h)\tau}K^\times A(\lambda+h)]^{-1}\circ [\id -e^{-\lambda\tau}K^{\times}A(\lambda)]\kappa(\psi_{\lambda,\tau})\\
        &\quad \quad-A(\lambda)\kappa(\psi_{\lambda,\tau})\\
        &=A(\lambda+h)\kappa\circ \kappa^{-1}[\id -e^{-(\lambda+h)\tau}K^\times A(\lambda+h)]^{-1}[\id -e^{-(\lambda+h)\tau}K^\times A(\lambda+h)\\
        &\quad \quad+e^{-(\lambda+h)\tau}K^\times A(\lambda+h)-e^{-\lambda\tau}K^{\times}A(\lambda)]\kappa(\psi_{\lambda,\tau})-A(\lambda)\kappa(\psi_{\lambda,\tau})\\
        & =[A(\lambda+h)-A(\lambda)]\kappa(\psi_{\lambda,\tau})+\mathcal{R}_{\tau}(\lambda+h)\kappa\circ \kappa^{-1}[e^{-(\lambda+h)\tau}K^\times A(\lambda+h)\\
        &\quad \quad-e^{-\lambda\tau}K^{\times}A(\lambda)]\kappa(\psi_{\lambda,\tau}).
    \end{aligned}
    $$
    From the holomorphy of $A(\lambda)\kappa(\psi)$ and $\kappa^{-1}K^{\times}A(\lambda)\kappa(\psi)$, we can obtain the continuity of $\mathcal{R}_\tau(\lambda)\kappa(\psi)$ in $\lambda$. The holomorphy of $\mathcal{R}_\tau(\lambda)\kappa(\psi)$ directly follows from this fact and the equicontinuity of $\{\mathcal{R}_\tau(\lambda)\kappa\}_{\mu\in V_\lambda}$. More precisely, we get 
    \begin{equation}\label{eq:resolvent_derivative}
    \begin{aligned}
        &\lim_{h\to 0}\langle\frac{1}{h}[\mathcal{R}_\tau(\lambda+h)\kappa(\psi)-\mathcal{R}_\tau(\lambda)\kappa(\psi)]|\varphi\rangle \\
        &= \langle \frac{\di}{\di \lambda}A(\lambda)\kappa(\psi_{\lambda,\tau})|\varphi\rangle+e^{-\lambda\tau}\langle \mathcal{R}_\tau(\lambda)\kappa\circ\frac{\di}{\di \lambda}\kappa^{-1}K^{\times}A(\lambda)\kappa(\psi_{\lambda,\tau})|\varphi\rangle \\
        &\quad\quad-\tau e^{-\lambda\tau}\langle \mathcal{R}_\tau(\lambda)K^{\times}A(\lambda)\kappa(\psi_{\lambda,\tau})|\varphi\rangle,
    \end{aligned}
    \end{equation}
    for each $\varphi\in X$.
\end{proof}
As shown in the preceding subsection, $\langle A(\lambda)\kappa(\psi)|\varphi\rangle$ provides an analytic continuation of $((\lambda- iH)^{-1}\psi,\varphi)$ across the brach cut $ iI\subset  i\mathbb{R}$. In particular, when $\Re(\lambda)>0$, we have $\langle A(\lambda)\kappa(\psi)|\varphi\rangle=((\lambda- iH)^{-1}\psi,\varphi)$. Similarly, for $\langle \mathcal{R}_\tau(\lambda)\kappa(\psi)|\varphi\rangle$, we have the following conclusion. 
\begin{proposition}\label{pro:generalized_resolvent_continuation}
    If $\Re (\lambda)>0 $ and $\lambda\in \hat{\varrho}(\mathcal{A})\cap\varrho(\mathcal{A})$, then we have $\langle \mathcal{R}_\tau(\lambda)\kappa(\psi)|\varphi\rangle=(\Delta(\lambda)^{-1}\psi,\varphi)$ for any $\psi,\varphi\in X$.
\end{proposition}
\begin{proof}
    For each $\psi\in X$ and $\varphi\in Y\subset X$, we have
    $$
    \begin{aligned}
        \langle (\id-e^{-\lambda\tau}K^{\times}A(\lambda)) \kappa(\psi)|\varphi\rangle&=\langle \kappa(\psi)|\varphi\rangle -e^{-\lambda\tau}\langle K^{\times}A(\lambda)\kappa(\psi)|\varphi \rangle\\
    &=(\psi,\varphi)-e^{-\lambda\tau}\langle A(\lambda)\kappa(\psi)|K^{\ast}\varphi \rangle\\
    &=(\psi,\varphi)-e^{-\lambda\tau}(K(\lambda- iH)^{-1}\psi,\varphi).\\
    \end{aligned}
    $$
Since $Y$ is a dense subspace, it follows that 
\begin{equation}
    \kappa^{-1}(\id-e^{-\lambda\tau}K^{\times}A(\lambda)\kappa(\psi)=(\id-e^{-\lambda\tau}K(\lambda- iH)^{-1}\psi
\end{equation}
for any $\psi\in X$. Let $\psi\in \text{Ran}(\id -e^{-\lambda\tau}K^{\times}A(\lambda))$. Then, for any $\varphi\in X$, we get 
\begin{equation}\label{eq:analytical_resolvent}
\begin{aligned}
    \langle \mathcal{R}_\tau(\lambda)\kappa(\psi)|\varphi\rangle&=\langle A(\lambda)\kappa\circ \kappa^{-1}(\id-e^{-\lambda\tau}K^{\times}A(\lambda))^{-1}\kappa(\psi)|\varphi\rangle\\
    &=((\lambda- iH)^{-1}(\id-e^{-\lambda\tau}K(\lambda- iH)^{-1})^{-1}\psi,\varphi)\\
    &=(\Delta(\lambda)^{-1}\psi,\varphi). 
\end{aligned}
\end{equation}
Since $\lambda\in \hat{\varrho}(\mathcal{A})$, we know that $\text{Ran}(\id -e^{-\lambda\tau}K^{\times}A(\lambda))$ is dense in $X$. Hence, (\ref{eq:analytical_resolvent}) holds for any $\psi\in X$, which completes the proof.
\end{proof}
The relation between the generalized retarded resolvent $\mathcal{R}_\tau(\lambda)$ and the bidual $\Delta(\lambda)^{\times}$ can be summarized as follows.
\begin{proposition}
    The following statements hold\\
    \noindent $\mathbf{(a)}$ For any $\psi\in X$, we have $\Delta(\lambda)^\times\circ\mathcal{R}_\tau(\lambda)\kappa(\psi)=\kappa(\psi);$\\
    \noindent $\mathbf{(b)}$ For any $\psi\in X^\prime$ such that $\Delta(\lambda)^\times\psi\in \kappa X$, we have $\mathcal{R}_\tau(\lambda)\circ\Delta(\lambda)^\times\psi=\psi.$
\end{proposition}
\begin{proof}
    Recall that $\Delta(\lambda)^\times:=\lambda- iH^\times-e^{-\lambda\tau}K^{\times}:X^\prime \to X^\prime$. From the proof of Proposition \ref{pro:generalized_eigenvalue}, we have seen that
    \begin{equation}\label{eq:Alambda_identity}
        (\lambda- iH^\times)\circ A(\lambda)\kappa(\psi)= \kappa(\psi)
    \end{equation}
    for any $\psi\in X$ and, in particular, we have $\text{Ran}(A(\lambda))\subset \text{Dom}(H^{\times})$. $\mathbf{(a)}$ directly follows from (\ref{eq:Alambda_identity}) and \textbf{(H7)}. Indeed, we have
    $$
    \begin{aligned}
        \Delta(\lambda)^\times\circ\mathcal{R}_\tau(\lambda)\kappa(\psi)&=(\lambda- iH^\times-e^{-\lambda\tau}K^{\times})A(\lambda)\circ(\id-e^{-\lambda\tau}K^\times A(\lambda))^{-1}\kappa(\psi)\\
        &=\kappa(\psi).
    \end{aligned}
    $$
    Similarly, we can obtain $\mathbf{(b)}$ by rewriting $\mathcal{R}_\tau(\lambda)$ as $(\id-e^{-\lambda\tau}A(\lambda)K^{\times})^{-1}\circ A(\lambda)$.
\end{proof}
\begin{remark}
    For any $\psi\in X$, it follows from Proposition \ref{pro:analytical_resolvent} and topological properties of $X$ (see \textbf{(H2)}) that $ \mathcal{R}_\tau(\lambda)\kappa(\psi)$ can be expressed as a Laurent series 
    \begin{equation}\label{eq:Laurent series}
        \mathcal{R}_\tau(\lambda)\kappa(\psi)=\sum_{j=-\infty}^{\infty}\frac{1}{(\lambda_0-\lambda)^j} E_j\kappa(\psi),
    \end{equation}
    near any $\lambda_0\in \hat{\varrho}(\mathcal{A})$. From (\ref{eq:resolvent_derivative}), it is clear that the series can converge in the weak dual topology. Since $X$ is barreled, the dual space $X^\prime$ equipped with the strong dual topology is quasi-complete, therefore, satisfying the convex envelope property \cite{SW,Treves}. In addition, any weakly bounded subset of $X^\prime$ is also bounded with respect to the strong dual topology. We can obtain that the weakly holomorphic function $\mathcal{R}_\tau(\lambda)\kappa(\psi)$ is also strongly holomorphic for any $\psi\in X$. Hence, the Laurent series in (\ref{eq:Laurent series}) converges in the strong dual topology too. Similarly, the Laurent expansion of $\mathcal{R}_\tau(\lambda)i(\psi)$ near each isolated singularity (i.e., isolated generalized eigenvalue) is also well defined. We refer \cite[Appendix A]{chibaB} for detailed discussions and proofs. The Laurent expansion (\ref{eq:Laurent series}) will play important roles later.
\end{remark}
\subsection{Properties of generalized spectrum}
In this subsection, we present several important properties of the generalized spectrum $\hat{\sigma}(\mathcal{A})$. The relations between the generalized spectrum and the spectrum $\sigma(\mathcal{A})$ in the classical sense will also be concluded.

Most importantly, like the result in \cite{chibaB}, it is shown that under some compactness condition of $\kappa^{-1}K^{\times}A(\lambda)\kappa:X\to X$, continuous singularities (i.e., the continuous spectrum of $\mathcal{A}$) disappear.

To avoid confusion, we first fix some topological concepts. 

Let $M$ and $N$ be two Hausdorff locally convex topological vector spaces. A linear operator $L:M\to N$ is called a bounded operator if it maps some neighborhood to a bounded set in $N$ \cite{SW}. Similarly, the linear operator $L$ is called a compact operator if it maps some neighborhood to a relatively compact subset of $N$. If $L$ is parameterized by $\lambda$ (i.e., $L:=L(\lambda)$), then $L(\lambda)$ is called bounded/compact uniformly in $\lambda$ if such a neighborhood is independent of $\lambda$.  
\begin{remark}
    There are several definitions for a bounded/compact operator on a topological vector space. In other literature (for example, \cite{Edw}), we can also define a bounded linear operator as a linear operator that maps every bounded set into a bounded set. In general, these definitions are far from equivalent. We refer \cite{hmz} for detailed discussions of this topic. 
\end{remark}
Applying Proposition \ref{pro:generalized_resolvent_continuation}, the relation between $\hat{\sigma}(\mathcal{A})$ and $\sigma({\mathcal{A}})$ on the right half-plane can be stated as:
\begin{theorem}\label{thm:cplus}
    When $\Re(\lambda)>0$, we have the following.\\
    \noindent $\mathbf{(a)}$ $\hat{\sigma}(\mathcal{A})\cap \{\Re(\lambda)>0\}\subset\sigma(\mathcal{A})\cap \{\Re(\lambda)>0\}$ and, in particular, $\hat{\sigma}_p(\mathcal{A})\cap \{\Re(\lambda)>0\}\subset\sigma_p(\mathcal{A})\cap\{\Re(\lambda)>0\}.$\\
    \noindent $\mathbf{(b)}$ If $\lambda$ is an isolated eigenvalue, then $\lambda$ is also an isolated eigenvalue in the generalized sense. 
\end{theorem}
\begin{proof}
    Let $\lambda$ ($\Re(\lambda)>0$) belong to the resolvent set in the usual sense: $\lambda\in \varrho(\mathcal{A})$. Because $\varrho(\mathcal{A})$ is open in $\mathbb{C}$, we can find a (bounded) neighborhood of $\lambda$, $V_\lambda$ s.t. $\Delta(\mu):\mathcal{H}\to \mathcal{H}$ has a continuous inverse for any $\mu\in V_\lambda$ and $\{\Delta(\mu)^{-1}\psi\}_{\mu\in V_\lambda}$ is bounded in $\mathcal{H}$ for each $\psi\in \mathcal{H}$. We can define $\langle \Delta(\mu)^{-1}\cdot|:X\to X^\prime$ by
   \begin{equation}
       \langle \Delta(\mu)^{-1}\psi|\varphi\rangle :=(\Delta(\mu)^{-1}\psi,\varphi). 
   \end{equation}
   Since the topology of $X$ (resp., $X^\prime$) is stronger (resp., weaker) than the topology of $\mathcal{H}$, it is clear that $\lambda\in \hat{\varrho}(\mathcal{A})$ and $\langle \mathcal{R}_\tau(\lambda)\kappa(\psi)|\varphi\rangle=(\Delta(\lambda)^{-1}\psi,\varphi)$.
   
   Now, let $\lambda$ be an generalized eigenvalue. From the definition, there exists some $\psi\in X$ s.t. $(\id-e^{-\lambda\tau}\kappa^{-1}K^\times A(\lambda)\kappa)\psi=0$. Since $H$ is self-adjoint, $(\lambda- iH)^{-1}$ exists for any $\lambda$ with $\Re(\lambda)\not=0$. We get
   $$
   \begin{aligned}
   0&=(\id-e^{-\lambda\tau}\kappa^{-1}K^\times A(\lambda)\kappa)(\id- iH)\circ(\id- iH)^{-1}\psi\\
   &=(\id- iH-e^{-\lambda\tau}\kappa^{-1}K^\times \kappa)\circ (\id- iH)^{-1}\psi\\
   &= \Delta(\lambda)\circ (\id- iH)^{-1}\psi.
   \end{aligned}
   $$
   Hence, $\lambda\in \sigma_p(\mathcal{A})$ and $(\id- iH)^{-1}\psi\in \mathcal{H}$ solves the characteristic equation. It completes the proof of $\mathbf{(a)}$.
   
   For $\mathbf{(b)}$, let us assume that $\lambda$ is an isolated eigenvalue of $\mathcal{A}$ and $\Delta(\lambda)\psi=0$ for some $\psi\in \mathcal{H}$. Then, there exists some small enough neighborhood of $\lambda$, $V_\lambda$ s.t. each $\mu\not=\lambda$ in $V_\lambda$ belongs to the resolvent set $\varrho(\mathcal{A})$. As shown in the beginning of this proof, we have $\mu\in \hat{\varrho}(\mathcal{A})$. If $\lambda\in \hat{\varrho}(\mathcal{A})$, a contradiction can be induced from Proposition \ref{pro:generalized_resolvent_continuation} and $\Delta(\lambda)\psi=0$.
\end{proof}
\begin{remark}
    The first statement in the above theorem illustrates that, compared with the spectrum in the usual sense, the generalized spectrum "shrinks" on the right complex half-plane. In general, the continuous spectrum, if exists, will not be concluded in the generalized spectrum. As shown in the proof, this is due to the "strong-weak" topology interplay induced by the Gelfand triple $X\subset\mathcal{H}\subset X^\prime$. 
    
    On the other hand, $\mathbf{(b)}$ in Theorem \ref{thm:cplus} suggests that different from the shrinking of the essential spectrum, those isolated singularities of $\Delta(\lambda)^{-1}$ will persist as isolated singularities on the Riemann surface induced by the generalized retarded resolvent $\mathcal{R}_\tau(\lambda)$.   
\end{remark}
\begin{remark}
    In most of the applications to linear evolution equations of Schrödinger or Friedrichs type \cite{hislop,gp}, the perturbation operator $K:\mathcal{H}\to \mathcal{H}$ is usually assumed to be at least a (relatively) compact operator (with respect to $H$). It follows from Kato's perturbation theory that $\sigma_{ess}(\mathcal{A})\cap \{\lambda|\Re(\lambda)>0\}=\emptyset$ \cite{kato}. Applying Theorem \ref{thm:cplus}, we know that
    $$
    \hat{\sigma}(\mathcal{A})\cap \{\lambda|\Re(\lambda)>0\}=\sigma_p(\mathcal{A})\cap\{\lambda|\Re(\lambda)>0\}.
    $$
\end{remark}
We are now in a position to investigate properties of the generalized spectrum on the left half-plane. The following lemma will play an important role in the sequel.
\begin{lemma}\label{lemma:pointwise_define}
    If there exists some $\lambda$-neighborhood $V_\lambda$ such that $\kappa^{-1}K^\times A(\mu)\kappa:X\to X$ is bounded uniformly in $\mu\in V_\lambda$, then, $\lambda\in \hat{\varrho}(\mathcal{A})$ if and only if $\id-e^{-\lambda\tau}\kappa^{-1}K^{\times}A(\lambda)\kappa:X\to X$ has a continuous inverse.
\end{lemma}
\begin{proof}
    Notice that for any $\tau>0$, $e^{-\lambda\tau}\kappa^{-1}K^\times A(\lambda)\kappa$ is holomorphic in $\lambda$. The proof of this lemma is almost identical to \cite[Proposition 3.18]{chibaB}. 
\end{proof}
\begin{remark}
    In the proof of \cite[Proposition 3.18]{chibaB}, we will use a theorem by Bruyn \cite{bruyn} about the existence of continuous inverse operator on Hausdorff topological vector spaces. It is well-known that when $X$ is a Banach space,  $\id-L:X\to X$ has a continuous inverse if $||L||<1$. Here, $||\cdot||$ denotes the operator norm and the inverse operator can be represented as the Neumann series $(\id-L)^{-1}=\sum_{i=0}^{\infty}L^{n}$. As a generalization, the Bruyn's theorem states \cite{bruyn}:
    \begin{theorem}
        Let $L:E\to E$ be a continuous operator on a sequentially complete locally convex space $E$. Assume that there exist a bounded set $B\subset E$ and $0$-neighborhood $V_0$ such that for given $\varepsilon>0$, there is a positive integer $n(\varepsilon)$ such that $K^{n(\varepsilon)}(V_0)\subset \varepsilon B$. Then, $\id-L$ has a continuous inverse on $E$.
    \end{theorem}
    By $\mathbf{(H2)}$, $X$ is quasi-complete barreled, hence sequentially complete \cite{SW}. Therefore, the above theorem can be applied in our proofs.
\end{remark}
In the classical spectral theory for linear operators in Banach spaces, descriptions of continuous spectrum sometimes associate with the so-called approximate spectrum. It is known that even though we cannot find a precise eigenfunction with respect to an element in the continuous spectrum, there exists a sequence of functions to approximate such a process.

In particular, the proof of existence of these approximate eigenfunctions is trivial when the space is normable. It is worth mentioning that we can also show such a thing in the generalized spectral theory on Gelfand triple. 

\begin{definition}\label{def:app}
    $\lambda$ is called a generalized approximate eigenvalue if there exists some net $(\psi_i)_{i\in \Lambda}\subset X$ such that 
    \begin{equation}
        (\id-e^{-\lambda\tau}\kappa^{-1}K^\times A(\lambda)\kappa)\psi_i\to 0.
    \end{equation}
 Denoted by $\hat{\sigma}_{app}(\mathcal{A})$, the collection of all approximate eigenvalues is called the approximate spectrum. 
\end{definition}
Applying Lemma \ref{lemma:pointwise_define} and the closed graph theorem on topological vector spaces, we have the following.
\begin{proposition}\label{thm:app_spectrum}
    Assume that $\kappa^{-1}K^\times A(\lambda)\kappa:X\to X$ is bounded, uniformly in $\lambda$ and $X$ is a complete space (a Fréchet space, for example). Then, for any $\tau\geq 0$, we have $\hat{\sigma}_c(\mathcal{A})\subset \hat{\sigma}_{app}(\mathcal{A})$. In other words, for any $\lambda\in \hat{\sigma}_c(\mathcal{A})$, there exists $(\psi_i)_{i\in \Lambda}\subset X$ such that $\lim_{i\in \Lambda}(\id-e^{-\lambda\tau}K^\times A(\lambda))\kappa(\psi_i)=0$.
\end{proposition}
\begin{proof}
    Lemma \ref{lemma:pointwise_define} implies that
    $$
    \hat{\sigma}(\mathcal{A})\backslash \hat{\sigma}_p(\mathcal{A})\subset \{\lambda|(\id-e^{-\lambda\tau}\kappa^{-1}K^\times A(\lambda)\kappa)^{-1}:X\to X  \ \text{is not continuous }\}.
    $$
    Since $\kappa^{-1}K^\times A(\lambda)\kappa$ is bounded on $X$ for each $\lambda$, by applying \cite[Lemma 3]{maeda}, we can obtain that ${\sigma_W}(\kappa^{-1}K^\times A(\lambda)\kappa)=\sigma_{C}(\kappa^{-1}K^\times A(\lambda)\kappa)$. Here $\sigma_W(\cdot):=\mathbb{C}\backslash\varrho_W{(\cdot)}$ denotes the spectrum of a linear operator in the sense of Waelbroeck \cite{wael} and $\sigma_C(\cdot):=\mathbb{C}\backslash\varrho_C{(\cdot)}$ denotes the spectrum of a linear operator in the classical sense. More precisely, we have 
    \begin{equation}
    \begin{aligned}
    \varrho_W{(L)}:=&\{\lambda|\text{There exists } V_\lambda\ \text{s.t. } (\mu-L)^{-1}\ \text{is continuous for }\mu\in V_\lambda\ \\
    &\text{and }\{(\mu-L)^{-1}\psi\}_{\mu\in V_\lambda}\ \text{is bounded}\};
    \end{aligned}
    \end{equation}
    \begin{equation}
    \varrho_C{(L)}:=\{\lambda|(\lambda-L)^{-1}\ \text{is continuous}\}.
    \end{equation}
    It follows from the completeness of $X$ that the closed graph theorem holds. Here, we mention that the assumption of $X$ in this theorem can be relaxed to any topological vector space, such that the closed graph theorem is satisfied. We refer a recent work \cite{noll} about this topic. 

    For any $\lambda\in \hat{\sigma}_c(\mathcal{A})$, we can obtain that $\text{Ran}(\id-e^{-\lambda\tau}\kappa^{-1}K^\times A(\lambda)\kappa)$ is a proper dense subset of $X$. In particular, it is not closed. There exists a net $(\psi_i)_{i\in \Lambda}\subset X$ such that $(\id-e^{-\lambda\tau}\kappa^{-1}K^\times A(\lambda)\kappa)\psi_i\to \tilde{\psi}$ and $\tilde{\psi}\not\in \text{Ran}(\id-e^{-\lambda\tau}\kappa^{-1}K^\times A(\lambda)\kappa)$. It is clear that $(\psi_i)_{i\in \Lambda}$ is not a Cauchy net. Hence, there exists a $0$-neighborhood, $U_0$ such that for any $i\in \Lambda$, there are $i_1,\ i_2\geq i$ satisfying $\psi_{i_1}-\psi_{i_2}\not\in U_0$. For the net $(\varphi_i)_{i\in \Lambda}:=(\psi_{i_1}-\psi_{i_2})_{i\in \Lambda}\subset X\backslash U_0$, we have 
    $$
    (\id-e^{-\lambda\tau}\kappa^{-1}K^\times A(\lambda)\kappa)\varphi_i=(\id-e^{-\lambda\tau}\kappa^{-1}K^\times A(\lambda)\kappa)(\psi_{i_1}-\psi_{i_2})\to 0.
    $$
    Consequently, we have $\lambda\in \hat{\sigma}_{app}(\mathcal{A})$ in the sense of Definition \ref{def:app}.
\end{proof}
The following theorem states that under compactness condition of $\kappa^{-1}K^\times A(\lambda)\kappa$, the generalized spectrum associated with the delay equation consists of only generalized eigenvalues ($\hat{\sigma}_c(\mathcal{A})=\hat{\sigma}_r(\mathcal{A})=\emptyset$). Inspired by Ringrose \cite{ringrose}, the proof is completed by associating $\kappa^{-1}K^\times A(\lambda)\kappa$ with a linear operator on a Banach space, so that the classical Riesz-Schauder theorem is applicable to show the disappearance of continuous/residual spectrum. Here we point out that different from the setting in Proposition \ref{thm:app_spectrum}, we do not need to additionally assume the completeness of $X$. 
\begin{theorem}\label{thm:riesz schauder}
    Assume that $\kappa^{-1}K^\times A(\lambda)\kappa:X\to X$ is compact, uniformly in $\lambda$. We have $\hat{\sigma}(\mathcal{A})=\hat{\sigma}_p(\mathcal{A})$. 
\end{theorem}
\begin{proof}
    There exists some (bounded) neighborhood $V$ such that $\kappa^{-1}K^\times A(\lambda)\kappa(V)$ is relatively compact for any $\lambda$. Induced from $V$, a continuous seminorm of $X$ can be defined by
    \begin{equation}
        p(x)=\inf\{|\lambda|\mid x\in \lambda V\}. 
    \end{equation}
    $M=\{x\in X\mid p(x)=0\}$ is a closed subspace of $X$ and $Z=X/M$ is a normed space endowed with a norm $P([x])=p(x)$. $[\cdot]=\cdot+M$ denotes each equivalence class of $x$. We denote by $Z_0$ a Banach space which is the completion of $Z$. A linear operator $Q(\tau,\lambda):Z\to Z$ is defined by 
    $$
    Q(\tau,\lambda)[x]=e^{-\lambda\tau}[\kappa^{-1}K^\times A(\lambda)\kappa(x)], \quad x\in X.
    $$
    We denote by $Q_0(\tau,\lambda):Z_0\to Z_0$ the continuous extension of $Q(\tau,\lambda)$. It is easy to verify that $Q_0(\tau,\lambda)$ is compact (see \cite{ringrose}). Henceforth, we denote $\kappa^{-1}K^\times A(\lambda)\kappa$ by $C(\lambda)$ for simplicity. 

\noindent \textbf{Step one.} Firstly, we show the one-to-one correspondence of eigenvalues of $e^{-\lambda\tau}C(\lambda)$ and $Q_0(\tau,\lambda)$. In particular, we can verify that if $\lambda$ is a generalized eigenvalue, then integer $1$ is an eigenvalue of $Q_0(\tau,\lambda)$, and vice versa. Indeed, if there exists nonzero $x\in X$ which satisfies  
$$
x-e^{-\lambda\tau}C(\lambda)x=0,
$$ 
we have 
$$
Q_0(\tau,\lambda)[x]=e^{-\lambda\tau}[C(\lambda)x]=[x].
$$
If $x\in M$, $\mu x\in M\subset V$ for any $\mu$. It contradicts with the boundedness of $V$. Here, from the compactness of $C(\lambda)$, we can also obtain that $C(\lambda)(M)=\{0\}$.

Conversely, it is easy to verify that $\text{Ran}(Q_0(\tau,\lambda))\subset Z$. If $1$ is an eigenvalue of $Q_0(\tau,\lambda)$, there exists $y\notin M$ such that $Q_0(\tau,\lambda)[y]=[y]$ and $C(\lambda)(y)-y\in M$. Hence $\lambda$ is a generalized eigenvalue with
$$
(\id-e^{-\lambda\tau}C(\lambda))C(\lambda)(y)=(\id-e^{-\lambda\tau}C(\lambda))\circ(C(\lambda)(y)-y+y)=0.
$$
\noindent \textbf{Step two.} Assume that $\lambda$ is not a generalized eigenvalue. From Lemma \ref{lemma:pointwise_define}, it suffices to prove $\id-e^{-\lambda\tau}C(\lambda):X\to X$ has a continuous inverse. Indeed, it is clear that $1$ is not an eigenvalue of $Q_0(\tau,\lambda):Z_0\to Z_0$. Since $Z_0$ is a Banach space and $Q_0(\tau,\lambda)$ is compact, Riesz-Schauder theorem implies that $\id-Q_0(\tau,\lambda)$ has a continuous inverse. Since 
$$
(\id-Q_0(\tau,\lambda))^{-1}=\id+Q_0(\tau,\lambda)(\id-Q_0(\tau,\lambda))^{-1},
$$
we have $(\id-Q_0(\tau,\lambda))^{-1}(Z)\subset Z$. For any $x\in X$, there exists some $y\in X$ such that $(\id-Q_0(\tau,\lambda))^{-1}([x])=[y]$. Hence,
$$
[x]=[y]-Q_0(\tau,\lambda)[y]=[y-e^{-\lambda\tau}C(\lambda)y]
$$
and $x-y+e^{-\lambda\tau}C(\lambda)y\in M$. For $x$ there exists $y_0=x+e^{-\lambda\tau}C(\lambda)y$ satisfying $(\id-e^{-\lambda\tau}C(\lambda))y_0=x$. Hence, $(\id-e^{-\lambda\tau}C(\lambda))^{-1}:X\to X$ exists and is fully defined. Now we are in a position to prove the continuity of the inverse operator. For $x\in X$, we have 
$$
(\id-Q_0(\tau,\lambda))^{-1}([x])=[y]=[y_0]=[(\id-e^{-\lambda\tau}C(\lambda))^{-1}x].
$$
Then
$$
\begin{aligned}
p((\id-e^{-\lambda\tau}C(\lambda))^{-1}x)=P((\id-Q_0(\tau,\lambda))^{-1}([x]))\leq ||\id-Q_0(\tau,\lambda)||P([x]).
\end{aligned}
$$
The continuity holds from this estimate and compactness of $C(\lambda)$. 
\end{proof}

\subsection{Asymptotic behavior on Gelfand triple}
Throughout this subsection, we investigate asymptotic behavior of solutions to the linear delay equation (\ref{eq:aFDE}). 

Applying the spectral properties proposed in the preceding section, the integral contour in the inverse Laplace transform can be deformed on a "better" Riemann surface induced by a Gelfand triple. In particular, we obtain exponential decay states of solution, where decay rate is determined by the generalized eigenvalues on the left complex half-plane. This is associated with the so-called Landau damping in the linear level. 

Recall that solution of (\ref{eq:aFDE}), $u(t)$ can be expressed using an inverse Laplace formula
\begin{equation}\label{eq:sec3_hilbert}
\begin{aligned}
    (u(t),\varphi)&=\frac{1}{2\pi i}\int_{\Gamma}e^{\lambda t}(\Delta(\lambda)^{-1}x,\varphi)\di \lambda+\frac{1}{2\pi i}\int_{\Gamma}e^{\lambda t}(\Delta(\lambda)^{-1}\hat{K}(\lambda-A_L)^{-1}f,\varphi)\di \lambda.
\end{aligned}
\end{equation}
Here we can find large enough $a>0$ such that $\Gamma=\{\lambda\in \mathbb{C}|\Re(\lambda)=a\}$. $(x,f)$ are initial data.

From Proposition \ref{pro:generalized_resolvent_continuation}, we know that the generalized retarded resolvent operator $\mathcal{R}_\tau(\lambda):iX\to X^\prime$ extends the retarded resolvent operator $\Delta(\lambda)^{-1}$ and $\langle\mathcal{R}_\tau(\lambda)i\psi|\varphi\rangle=(\Delta(\lambda)^{-1}\psi,\varphi)$ holds on $\mathbb{C}_+$. We get 
\begin{equation}\label{eq:sec3_gelfand}
\begin{aligned}
     (u(t),\varphi)=\langle \kappa(u(t))|\varphi\rangle &= \frac{1}{2\pi i}\int_{\Gamma}e^{\lambda t}\langle \mathcal{R}_\tau(\lambda)\kappa(x)|\varphi\rangle\di \lambda\\
    &\quad+\frac{1}{2\pi i}\int_{\Gamma}e^{\lambda t}\langle \mathcal{R}_\tau(\lambda)\kappa\circ\hat{K}(\lambda-A_L)^{-1}f|\varphi\rangle\di \lambda,
\end{aligned}
\end{equation}
for any $x,\varphi\in X$. Here, because
$$\hat{K}(\lambda-A_L)^{-1}f=\int_{-\tau}^{0}e^{-\lambda(\tau+s)}Kf(s)\di s ,$$
the initial condition $f$ is assumed to satisfy $Kf(s)\in X$ for any $s\in [-\tau,0]$ so that (\ref{eq:sec3_gelfand}) is well defined. For any $x\in X$, recall that the generalized retarded resolvent operator $\mathcal{R}_\tau(\lambda)$ has Laurent series expansions near any isolated singularity $\lambda_0$ given by 
$$
\mathcal{R_\tau}(\lambda)\kappa(x)=\sum_{j=-\infty}^{+\infty}\frac{1}{(\lambda_0-\lambda)^j}E_j\kappa(x).
$$
The operators $E_j$ are $X^\prime$-valued and independent of $\lambda$. The above series converges in the strong dual topology. 

Let $\lambda_0$ be a generalized eigenvalue and $\Gamma_0$ be a simple closed curve that encloses $\lambda_0$. By the residue theorem of weak$^\ast$ Pettis integrals, we obtain 
\begin{equation}
    \frac{1}{2\pi  i}\int_{\Gamma_0}e^{\lambda t} \mathcal{R}_\tau(\lambda)\kappa(x)\di \lambda=-\sum_{j=0}^{\infty}\frac{(-t)^j}{j!}e^{\lambda_0 t}E_{j+1}\kappa(x),
\end{equation}
for any $x\in X$.

Under basic setting $\textbf{(H1-6)}$ and the compactness assumption, it follows from Theorem \ref{thm:riesz schauder} that singularities of $\mathcal{R}_\tau(\lambda) \kappa$ are (isolated) generalized eigenvalues $\hat{\sigma}_p(\mathcal{A})$. In particular, continuous singularities on the Riemann surface induced by $(\lambda- iH)^{-1}$ disappear in the generalized sense. Combined with these findings, we can obtain the following corollary of Theorem \ref{thm:riesz schauder}.
\begin{corollary}\label{coro:eigenexpansion}
    Assume that $\kappa^{-1} K^\times A(\lambda)\kappa:X\to X$ is compact, uniformly in $\lambda$. Then we have the following representation of solution.
    \begin{equation}\label{eq:eigenexpansion}
\begin{aligned}
    (u(t),\psi)
    &=\frac{1}{2\pi i}\int_{\Gamma^\prime}e^{\lambda t}\langle \mathcal{R}_\tau(\lambda)\kappa(x+\hat{K}(\lambda-A_L)^{-1}f)|\psi\rangle\di \lambda \\
    &\quad-\sum_{k=1}^{\infty}\sum_{j=0}^{\infty}\sum_{m=0}^{j}\frac{(-t)^{j-m}}{(j-m)!}e^{\lambda_k t}\langle E_{j+1}^k\kappa(\hat{K}(\lambda_k-A_L)^{-m-1}f)|\psi\rangle\\
    &\quad-\sum_{k=1}^{\infty}\sum_{j=0}^{\infty}\frac{(-t)^j}{j!}e^{\lambda_k t}\langle E_{j+1}^k \kappa(x)|\psi\rangle.
\end{aligned}
\end{equation}
Here, $\Gamma_1$ lies on the left half-plane. $\{\lambda_k\}_{k\in\mathbb{N}}$ are generalized eigenvalues. 
\end{corollary}
\begin{proof}
    (\ref{eq:eigenexpansion}) directly follows from contour deformation and residue theorem. 
\end{proof}
\section{Linear stability of a coupled oscillator system with time delay}\label{sec:4}
In this section, the linear stability analysis of incoherent state of the infinite dimensional Kuramoto model with a single time lag is implemented based on the generalized spectral theory on a Gelfand triple. 

As shown in the Introduction, the governing differential equation reads
\begin{equation}\label{eq:km3}
    \begin{cases}
    \partial _t \rho+\partial_\theta(V[\rho]\rho)=0,\\
     V[\rho]=\omega+k\int_{-\infty}^{+\infty}\int_0^{2\pi}\sin(\theta_\ast-\theta)\rho(t-\tau,\omega_\ast,\theta_\ast)g(\omega_\ast)\di \theta_\ast \di \omega_\ast.\\
    \end{cases}
\end{equation}
The complex order parameter (centroid of phase oscillators) of each solution is defined by
$$
r(t)=\int_{-\infty}^{+\infty}\int_0^{2\pi}e^{ i\theta}\rho(t,\omega,\theta)g(\omega)\di \theta \di \omega,
$$
and $r(t)=0$ corresponds to the incoherent (completely de-synchronous) state. 

We will show that $r(t)$ for the linearized equation decays exponentially to zero in the weak coupling regime and nearly identical frequency setting. 
\subsection{Linearization at the incoherent state}
The Fourier series expansion of $\rho(t,\omega,\theta)$ reads
$$
\begin{aligned}
    \rho(t,\omega,\theta)&=\frac{1}{2\pi}\sum_{j=-\infty}^{+\infty}F_j(t,\omega)e^{i\theta}
    =\frac{1}{2\pi}\sum_{j=-\infty}^{+\infty}[\int_0^{2\pi}\rho(t,\omega,\theta_\ast)e^{- ij\theta_\ast}\di \theta_\ast]e^{i\theta}.
\end{aligned}
$$
By substituting the expansion into (\ref{eq:km3}), we obtain that for each $j\in \mathbb{Z}$, the Fourier coefficient $F_{-j}(t,\omega)$ satisfies
\begin{equation}
\begin{aligned}
    \partial_tF_{-j}=ij\omega F_{-j}(t,\omega)+\frac{jk}{2}r(t-\tau)F_{-j+1}(t,\omega)-\frac{jk}{2}\overline{r(t-\tau)}F_{-j-1}(t,\omega).
\end{aligned}
\end{equation}
In particular, 
$$
\partial_t F_{-1}(t,\omega)= i\omega F_{-1}(t,\omega)+\frac{k}{2}r(t-\tau)-\frac{k}{2}\overline{r(t-\tau)}F_{-2}(t,\omega).
$$
For the incoherent state, we have $Z_j(t,\omega)\equiv0$ for $j\not=0$. It is worth noticing that the delayed mean-field $r(t-\tau)$ satisfies
$$
r(t-\tau)=\int_{-\infty}^{+\infty}F_{-1}(t-\tau,\omega)g(\omega)\di \omega.
$$
Hence, the linear stability of the incoherent state is equivalent to the stability of zero solution of the following linear evolution equation with a single time delay:
        $$
        \partial_t u(t,\omega)= i\omega  u(t,\omega)+\frac{k}{2}\int_{-\infty}^{+\infty}u(t-\tau,\omega)g(\omega)\di \omega.
        $$
Equivalently, we are now in a position to analyze the stability of the following equation
\begin{equation}\label{eq:km4}
    \frac{\di}{\di t}u(t)= i\mathcal{M}u(t)+\frac{k}{2}\mathcal{P}u(t-\tau),
\end{equation}
where $u:[-\tau,\infty)\to L^2(\mathbb{R},g(\omega)\di \omega)$. The multiplication operator $\mathcal{\mathcal{M}}:\psi(\omega)\mapsto \omega \psi(\omega)$ is an unbounded self-adjoint operator with absolutely continuous spectral measure. The integral operator $\mathcal{P}:\psi(\omega)\mapsto \int_{-\infty}^{+\infty}\psi(\omega)g(\omega)\di \omega$ is compact on the weighted Hilbert space $L^2(\mathbb{R},g(\omega)\di \omega)$ (henceforth, $\mathcal{H}$).

The coupling strength $k$ serves as a bifurcation parameter. It will be shown that under nearly identical frequency assumption, for any $\tau>0$, there exists some critical coupling strength $k_c(\tau)\in (-\infty,+\infty)$. When $|k|<|k_c(\tau)|$, the incoherent solution is stable and decay states of (\ref{eq:km4}) can be found in a weak sense. 
\subsection{Obstacles caused by continuous spectrum}
Let us investigate eigenvalues associated with the linear evolution equation (\ref{eq:km4}) at first. 
\begin{lemma}
    The eigenvalues satisfy the following integral equation
    \begin{equation}\label{eq:integral}
        e^{-\lambda\tau}\frac{k}{2}\int_{-\infty}^{+\infty}\frac{g(\omega)}{\lambda- i\omega}\di \omega=1.
    \end{equation}
\end{lemma}
\begin{proof}
    Denote by $\mathbb{I}\equiv 1$ a constant function in $\mathcal{H}=L^2(\mathbb{R},g(\omega)\di \omega)$. If $\lambda\in \mathbb{C}$ is an eigenvalue associated with (\ref{eq:km4}), there exists some $\psi\in \mathcal{H}$ which solves the following characteristic equation 
    $$
        \Delta(\lambda)\psi=(\lambda- i\omega)\psi(\omega)-e^{-\lambda\tau}\frac{k}{2}\int_{-\infty}^{+\infty}\psi(\omega)g(\omega)\di \omega=0.
    $$
    It implies that
    $$
    \psi=e^{-\lambda\tau}\frac{k}{2(\lambda- i\omega)}(\psi,\mathbb{I})\mathbb{I}.
    $$
    Then, we have
    $$
    1=e^{-\lambda\tau}\frac{k}{2}(\frac{1}{\lambda- i\omega},\mathbb{I})=e^{-\lambda\tau}\frac{k}{2}\int_{-\infty}^{+\infty}\frac{g(\omega)}{\lambda- i\omega}\di \omega.
    $$
\end{proof}
Let us investigate a limiting case where oscillators on $\mathbb{T}^1$ have identical natural frequency $\omega_0$ \cite{yeung}. In other words, we assume $g(\omega)=\delta(\omega-\omega_0)$. Then, (\ref{eq:integral}) is reduced to the following transcendental equation
\begin{equation}\label{eq:lambert}
    \lambda- i\omega_0-\frac{k}{2}e^{-\lambda\tau}=0.
\end{equation}
\begin{remark}
    Let $g(\omega)$ be a Gaussian distribution given by
    $$g(\omega)=\delta_h(\omega-\omega_0):=\sqrt{\frac{h}{\pi}}e^{-h(\omega-\omega_0)^2}.$$
    By using changing of variable, we have
    $$
    \sqrt{\frac{h}{\pi}}\int_{-\infty}^{+\infty}\frac{e^{-h(\omega-\omega_0)^2}}{\lambda- i\omega}\di \omega =\frac{1}{\sqrt{\pi}}\int_{-\infty}^{+\infty}\frac{e^{-\omega^2}}{\lambda- i(\frac{\omega}{\sqrt{h}}+\omega_0)}\di \omega.
    $$
    Since $|\lambda- i(\frac{\omega}{\sqrt{h}}+\omega_0)|\geq|\Re(\lambda)|$, it follows from the dominated convergence theorem that the above integral converges to $\frac{1}{\lambda- i\omega_0}$ as $h\to\infty$.
\end{remark}
    \begin{remark}
        Moreover, we can obtain a more accurate estimate of the integral in (\ref{eq:integral}). Indeed, we have
    $$
    \frac{1}{\lambda- i(\frac{\omega}{\sqrt{h}}+\omega_0)}=\frac{1}{\lambda- i\omega_0}\cdot\frac{1}{1-\frac{ i\omega}{\sqrt{h}(\lambda- i\omega_0)}}.
    $$
    When $h>>1$, the above fraction can be expanded as
    $$
    \frac{1}{\lambda- i(\frac{\omega}{\sqrt{h}}+\omega_0)}=\frac{1}{\lambda- i\omega_0}\sum_{n=0}^\infty(\frac{ i\omega}{\sqrt{h}(\lambda- i\omega_0)})^n.
    $$
 Hence, it follows that
    $$
    \begin{aligned}
        \int_{-\infty}^{+\infty}\frac{\delta_h(\omega-\omega_0)}{\lambda- i\omega}\di \omega&=\frac{1}{\sqrt{\pi}}\int_{-\infty}^{+\infty}\frac{e^{-\omega^2}}{\lambda- i(\frac{\omega}{\sqrt{h}}+\omega_0)}\di \omega\\
        &=\frac{1}{\sqrt{\pi}(\lambda-i\omega_0)}\int_{-\infty}^{+\infty}e^{-\omega^2}\sum_{n=0}^\infty(\frac{ i\omega}{\sqrt{h}(\lambda- i\omega_0)})^n\di \omega\\
        &=\frac{1}{\lambda-i\omega_0}(1+0-\frac{1}{2h(\lambda- i\omega_0)^2}+0+...),\\
    \end{aligned}
    $$
    which yields 
    \begin{equation}\label{eq:estimate}
        \int_{-\infty}^{+\infty}\frac{\delta_h(\omega-\omega_0)}{\lambda- i\omega}\di \omega=\frac{1}{\lambda-i\omega_0}(1-\frac{1}{2h(\lambda- i\omega_0)^2}+\mathcal{O}(\frac{1}{h^2})).
    \end{equation}
    \end{remark} 
It is known that the transcendental equation of type
\begin{equation}\label{eq:te}
\lambda+\alpha-\beta e^{-\lambda\tau}=0
\end{equation}
can be studied through the Lambert $W$ function \cite{CGHJK}.
\begin{definition}
    The Lambert W function is defined to be the multi-valued inverse of the function $z\mapsto z e^{z}$. Denote by $W(\zeta)$, the collection of all $z\in \mathbb{C}$ satisfying $z e^{z}=\zeta$.
\end{definition}
By studying the "graph-like" expressions of the complex branches of the Lambert $W$ function, we have the following necessary and sufficient conditions ensuring all roots of (\ref{eq:te}) have negative real parts.
\begin{theorem}[Theorem 1.2 of \cite{nishi}]\label{thm:nishi}
    Let $\alpha\in \mathbb{C}$, $\beta\in \mathbb{C}\backslash\{0\}$, and $\tau>0$. Then all roots of (\ref{eq:te}) have negative real parts if and only if either of the following two conditions holds:\\
    \noindent $\mathbf{(a)}$ $\Re(\alpha)>|\beta|;$\\
    \noindent $\mathbf{(b)}$ $-|\beta|<\Re(\alpha)\leq |\beta|$ and 
    \begin{equation}\label{eq:inequi}
        \arccos\{\cos[(\Im(\alpha)\tau+\mathrm{Arg}(\beta))]\}>\arccos(\frac{\Re(\alpha)}{|\beta|})+\tau\sqrt{|\beta|^2-\Re(\alpha)^2}.
    \end{equation}
\end{theorem}
It is clear that the set of roots of (\ref{eq:lambert}) can be rewritten as 
$$\frac{1}{\tau}W(\frac{k}{2}\tau e^{- i\omega_0\tau})+ i\omega_0.$$
Applying Theorem \ref{thm:nishi} with $\alpha=- i\omega_0$ and $\beta=\frac{k}{2}$, we have the following.
\begin{corollary}\label{coro:critical_value}
    For each $\tau>0$ and $\omega_0\not=0$, there exists critical coupling strength $k_c(\tau)$ given by
\begin{equation}
    k_c(\tau)=\frac{2}{\tau}\arccos(\cos(\omega_0\tau))-\frac{\pi}{\tau},
\end{equation}
such that for $|k|<|k_c(\tau)|$, all roots of (\ref{eq:lambert}) have negative real parts, and vice versa.
\end{corollary}
\begin{proof}
    When $\alpha=- i\omega_0$ and $\beta=\frac{k}{2}$, the inequality (\ref{eq:inequi}) can be rewritten as 
    $$
    \arccos[\cos(-\omega_0\tau+\mathrm{Arg}(\frac{K}{2}))]>\frac{\pi}{2}+\frac{\tau|k|}{2}.
    $$
    Hence, the positive critical value $k_c^+(\tau)$ is given by $\max\{0,\frac{2}{\tau}\arccos(\cos(\omega_0\tau))-\frac{\pi}{\tau}\}$ and the negative critical value $k_c^-(\tau)$ is given by
    $\min\{0,-\frac{2}{\tau}\arccos(-\cos(\omega_0\tau))+\frac{\pi}{\tau}\}$. The function $x\mapsto \arccos(\cos(x))$ is a $2\pi$-periodic and $\arccos(\cos(x))=|x|$ when $-\pi\leq x\leq \pi$. It is easy to verify that 
    $$
    \arccos(-\cos(x))=\pi-\arccos(\cos(x)).
    $$
    Thus, we have 
    $$
    \frac{2}{\tau}\arccos(\cos(\omega_0\tau))-\frac{\pi}{\tau}=-\frac{2}{\tau}\arccos(-\cos(\omega_0\tau))+\frac{\pi}{\tau},
    $$
    which completes the proof.
\end{proof}
The relation between the roots of (\ref{eq:integral}) with $g(\omega)=\delta_h(\omega-\omega_0)$ and (\ref{eq:lambert}) is stated as follows. We can verify that the solution set of the integral equation converges to the solution set of the transcendental equation in Hausdorff distance. Moreover, a one-to-one correspondence of roots of the two equations exists for sufficiently large $h$.
\begin{proposition}\label{prop:correspondance}
    Assume that roots of (\ref{eq:lambert}) have nonzero real parts. For any $\varepsilon>0 $, there exists $h^\ast$ such that if $h>h^\ast$, roots of (\ref{eq:integral}) are in $\varepsilon$-neighborhood of roots of (\ref{eq:lambert}). Conversely, if $\lambda$ solves (\ref{eq:lambert}), we can find some $\varepsilon>0 $ and $h^\ast$ such that for any $h>h^\ast$, there exists exactly one $\lambda^\prime$ with $|\lambda^\prime-\lambda|<\varepsilon$, which satisfies (\ref{eq:integral}).
\end{proposition}
\begin{proof}
    We can obtain that
    $$
    \begin{aligned}
        |\frac{1}{\lambda- i\omega}-\frac{1}{\lambda- i\omega_0}|=\frac{| i(\omega-\omega_0)|}{|(\lambda- i\omega)(\lambda- i\omega_0)|}\leq \frac{|\omega-\omega_0|}{\Re(\lambda)^2}.
    \end{aligned}
    $$
    Now, denote by $I(\lambda,h)$ the left part of (\ref{eq:integral}) with $g(\omega)=\delta_h(\omega-\omega_0)$
    $$
    I(\lambda,h):=e^{-\lambda \tau}\frac{k}{2}\sqrt{\frac{h}{\pi}}\int_{-\infty}^{+\infty}\frac{e^{-h(\omega-\omega_0)^2}}{\lambda- i\omega}\di \omega.
    $$
    Denote by $G(\lambda)$, the point-wise limit of $I(\lambda,h)$
    $$
    G(\lambda):=\frac{k}{2}\frac{e^{-\lambda\tau}}{\lambda- i\omega_0}.
    $$
    It follows that
    $$
    \begin{aligned}
        |I(\lambda,h)-G(\lambda)|&\leq \frac{|k|}{2}\frac{e^{-\Re(\lambda)\tau}}{\Re(\lambda)^2}\sqrt{\frac{h}{\pi}}\int_{-\infty}^{+\infty}e^{-h(\omega-\omega_0)^2}|\omega-\omega_0|\di \omega\\
        &=\frac{|k|}{2}\sqrt{\frac{1}{\pi h}}\frac{e^{-\Re(\lambda)\tau}}{\Re(\lambda)^2}.
    \end{aligned}
    $$
    For any $d>0$, it is clear that $I(\lambda,h)$ converges to $G(\lambda)$ uniformly in $\{\lambda|\Re(\lambda)>d\}$ and any compact subset of $\{\lambda|\Re(\lambda)<-d\}$.
    
    Firstly, let us investigate roots with negative real parts. It follows from (\ref{eq:estimate}) that
    \begin{equation}
        I(\lambda,h)-G(\lambda)=-\frac{G(\lambda)}{2h(\lambda- i\omega_0)^2}+\mathcal{O}(\frac{1}{h^2}).
    \end{equation}
    Let $\lambda_p$ ($p=1,2,3...$) be roots of (\ref{eq:lambert}) and $D_p$ be a $\lambda_p$-neighborhood with a fixed radius $\varepsilon>0$. It is easy to verify that $|G^\prime(\lambda_p)|\to \tau$ as $|\lambda_p|\to \infty$. Thus, when $|\lambda_p|$ is sufficiently large, we have an estimate 
    \begin{equation}
    |I(\lambda,h)-G(\lambda)|\leq \frac{C}{2h|\lambda-i\omega_0|},
    \end{equation}
    for $\lambda\in \partial D_p$ and some constant $C$. Consequently, there exists $h^\ast$ such that if $h>h^\ast$, we have 
    $$
    |I(\lambda,h)-G(\lambda)|<|G(\lambda)-1|
    $$
    for any $\lambda\in D_p$ ($p=1,2,3...$). We claim that every root of $G(\lambda)=1$ is a simple root. It follows from Rouché theorem that $I(\lambda,h)=1$ has exactly one simple root $\lambda_p(h)$ on each $D_p$ and $\lambda_p(h)\to \lambda_p$ as $h\to \infty$. 

    Next, we are in a position to prove that for sufficiently large $h$, $I(\lambda,h)=1$ has no additional root outside the $\lambda_p$-neighborhood. If this statement is wrong, there exist $\{(\lambda_n,h_n)\}_{n\in \mathbb{N}}$ with $h_n\to \infty$ such that $I(\lambda_n,h_n)=1$ and $\lambda_n\not\in\bigcup_p{D_p}$. If $\{\lambda_{n}\}_{n\in \mathbb{N}}$ is bounded, there exists a subsequence $\{\lambda_{n_j}\}_{j\in \mathbb{N}}$ which converges to some $\lambda^\ast$. Since $I(\lambda,h)$ converges to $G(\lambda)$ uniformly on any compact set, we can obtain $G(\lambda^\ast)=1$, which derives a contradiction. Now, we assume that $\{\lambda_{n}\}_{n\in \mathbb{N}}$ is an unbounded set ($|\lambda_n|\to \infty$). Since $|I(\lambda,h)|=|G(\lambda)|+\mathcal{O}(\frac{1}{h})$ and $|G(\lambda)|$ tends to either $\infty$ or $0$ as $|\lambda|\to \infty$. It contradicts with the unboundedness of $\{\lambda_{n}\}_{n\in \mathbb{N}}$.

    Now, we focus on roots on $\{\lambda|\Re(\lambda)>d\}$. It is well known that $G(\lambda)=1$ has at most finite roots on such region, so that we can find some compact simple curve to enclose all of the potential roots. The correspondence relation follows from Rouché theorem and uniform convergence of $I(\lambda,h)$ on $\{\lambda|\Re(\lambda)>d\}$. 
\end{proof}
\begin{corollary}\label{coro:transition_point}
    We assume that $|k|<|k_c(\tau)|$. Then, there exists some $h^\ast$ such that all the eigenvalues associated with (\ref{eq:km4}) with $g(\omega)=\delta_h(\omega-\omega_0)$ have negative real parts for any $h>h^\ast$.
\end{corollary}
Nevertheless, similar to the delay-free case \cite{sm,chibaA}, it is known that because of the appearance of continuous spectrum on the imaginary axis, we meet obstacles in obtaining a decay state even though no eigenvalues can be found on the right complex semi-plane.
\begin{lemma}\label{lemma:continuous_spectrum}
    The continuous spectrum associated with the linear evolution equation (\ref{eq:km4}) is given by
    \begin{equation}
    \sigma_c(\mathcal{A})=\sigma_c( i\mathcal{M})= i\mathbb{R}.
    \end{equation}
\end{lemma}
\begin{proof}
    Notice that the retarded resolvent operator associated with (\ref{eq:km4}) reads
    \begin{equation}\label{eq:km_retarded_resolvent}
    \Delta(\lambda)^{-1}=(\lambda- i\mathcal{M}-\frac{k}{2}e^{-\lambda\tau}\mathcal{P})^{-1}:\mathcal{H}\to \mathcal{H}.
    \end{equation}
    Since $\mathcal{P}$ is a compact operator, the statement of this lemma comes from Kato's perturbation theory about stability of essential spectrum under relatively compact perturbation \cite{kato}.
\end{proof}
As shown in the previous sections, by applying the inverse Laplace transform, the solution of (\ref{eq:km4}) can be expressed as 
\begin{equation}\label{eq:km_IL}
    u(t)=\frac{1}{2\pi  i}\lim_{b\to\infty}\int_{a- ib}^{a+ ib}e^{\lambda t}\Delta(\lambda)^{-1}(x+\frac{k}{2}\hat{\mathcal{P}}(\lambda-A_L)^{-1}f)\di \lambda.
\end{equation}
Here, $a$ is chosen to be sufficiently large such that the integral contour $\{z\in \mathbb{C}|\Re(z)=a\}$ separates the spectrum from the resolvent set. $\hat{\mathcal{P}}=\int_{-\tau}^{0}\di \eta:L^2([-\tau,0];\mathcal{H})\to \mathcal{H}$ with 
\begin{equation}
\eta(s)=
    \begin{cases}
        0, \quad \text{others},\\
    \mathcal{P},\quad s=-\tau.
    \end{cases}
\end{equation}
The spectrum associated with the linear evolution equation (\ref{eq:km4}) (singularities of $\Delta(\lambda)^{-1}$) consists of two components. One encompasses discrete eigenvalues that determined from the integral equation (\ref{eq:integral}). The other component refers to the continuous spectrum on the entire imaginary axis. 

For any $\tau>0$, when $|k|>|k_c(\tau)|$, it is easy to verify the incoherent state of (\ref{eq:km3}) with highly concentrated frequency distribution($g(\omega)=\delta_h(\omega-\omega_0)$, $h>>1$) has linear instability in the classical sense. Indeed, we can find some eigenvalue $\lambda_0$ with positive real part. Then the integral (\ref{eq:km_IL}) diverges as $t\to\infty$ because $\Re(\lambda_0)>0$.

However, in the weak coupling regime ($|k|<|k_c(\tau)|$), the integral contour in (\ref{eq:km_IL}) can not be deformed to enclose the "stable eigenvalues" on the left half-plane. In this case, the incoherent state is neutral.

\subsection{Linear stability on Gelfand triple}
As shown in the previous subsection, due to the continuous spectrum provided by the multiplication operator $ i\mathcal{M}:\mathcal{H}\to\mathcal{H}$, it is problematic to show linear stability of the incoherent solution even if there does not exist any eigenvalue with positive real part. 

In this subsection, we apply the generalized spectral theory of delayed linear evolution equations as introduced in Section \ref{sec:3} to (\ref{eq:km4}). By selecting a proper Fréchet-Montel space $\mathrm{Exp}$ decorated with stronger topology than $\mathcal{H}$, we verify that these continuous singularities disappears from the new Riemann surface induced by the Gelfand triple 
\begin{equation}
    \mathrm{Exp}\subset \mathcal{H}\subset{\mathrm{Exp}^\prime}.
\end{equation}
Then, decay states of (\ref{eq:km4}) directly come from the generalized eigenvalues (resonance poles) on the left half-plane. 

We start from defining the analytic continuation of the resolvent of $ i\mathcal{M}$. It is clear that 
\begin{equation}
    ((\lambda- i\mathcal{M})^{-1}\psi,\varphi)=\int_{-\infty}^{+\infty}\frac{1}{\lambda- i\omega}\psi(\omega)\varphi(\omega)g(\omega)\di \omega,
\end{equation}
for any $\psi,\varphi\in \text{Dom}(\mathcal{M})\subset{\mathcal{H}}$.

Denote by $\mathrm{Exp}(\beta,n)$, a Banach space containing holomorphic functions on $\{\lambda\in \mathbb{C}|\Im(\lambda)\geq -\frac{1}{n}\}$ endowed with a norm $$
||\psi||_{\beta,n}:=\sup_{\Im(\lambda)\geq-\frac{1}{n}}e^{-\beta|z|}|\psi(z)|.
$$
By taking the inductive limit with respect to both $n\in \mathbb{N^+}$ and $\beta\in \mathbb{N}$, we obtain a Fréchet-Montel space $\mathrm{Exp}$ (see \cite{chibaA}):
\begin{equation}
\mathrm{Exp}:=\bigcup_{\beta\in \mathbb{N}}\mathrm{Exp}(\beta)=\bigcup_{\beta\in \mathbb{N}}(\bigcup_{n\in \mathbb{N}^+}\mathrm{Exp}(\beta,n)).
\end{equation}
$\mathrm{Exp}$ describes a function space of holomorphic functions near the upper complex semi-plane that can grow at most exponentially. As proved in (\cite[Proposition 5.3]{chibaA}), $\mathrm{Exp}$ is a dense subspace of $\mathcal{H}$ with a stronger topology. Hence, 
$$
\mathrm{Exp}\subset \mathcal{H}\subset{\mathrm{Exp}^\prime}
$$
forms a Gelfand triple in the sense of Definition \ref{def:Gelfand triple}. We obtain a linear operator $A(\lambda) i:\mathrm{Exp}\to \mathrm{Exp}^{\prime}$ given by
$$
\langle A(\lambda) \kappa(\psi)|\varphi\rangle:= 
    \begin{cases}
        \int_{-\infty}^{+\infty}\frac{1}{\lambda- i\omega}\psi(\omega)\varphi(\omega)g(\omega)\di \omega+2\pi \psi(\frac{\lambda}{ i})\varphi(\frac{\lambda}{ i})g(\frac{\lambda}{ i}),\ \Re(\lambda)<0,\\
    \lim_{\Re(\lambda)\to0+}\int_{-\infty}^{+\infty}\frac{1}{\lambda- i\omega}\psi(\omega)\varphi(\omega)g(\omega)\di \omega,\ \lambda\in \mathbb{R},\\
        \int_{-\infty}^{+\infty}\frac{1}{\lambda- i\omega}\psi(\omega)\varphi(\omega)g(\omega)\di \omega,\ \Re(\lambda)>0,\\
    \end{cases}
$$
for each $\psi,\varphi\in \mathrm{Exp}$. It provides a analytic continuation of $((\lambda- i\mathcal{M})^{-1}\psi,\varphi)$ across the branch cut $ i\mathbb{R}$. 

We can verify that $\textbf{(H1-7)}$ are satisfied with $H=\mathcal{M}$ and $K=\frac{k}{2}\mathcal{P}$. The bi-dual operator $\mathcal{P}^{\times}:\mathrm{Exp}^{\prime}\to\mathrm{Exp}^{\prime}$ is given by 
\begin{equation}\label{eq:p_times}
    \langle \mathcal{P}^{\times}\psi|\varphi\rangle=\langle\psi|\mathbb{I}\rangle\langle \mathbb{I}|\varphi\rangle,
\end{equation}
where $\mathbb{I}\equiv1\in \mathrm{Exp}$, and the range $\text{Ran}(\mathcal{P}^{\times})$ is included in $i\mathrm{Exp}$.

In the sense of Definition \ref{def:resonance_pole}, generalized eigenvalues (resonance poles) associated with (\ref{eq:km4}) are determined as follows.
\begin{lemma}\label{lemma:resonance_pole}
    The generalized eigenvalues are given as roots of the following integral equation
    \begin{equation}\label{eq:resonance_pole}
    1=
        \begin{cases}
            \frac{k}{2}e^{-\lambda\tau}\int_{-\infty}^{+\infty}\frac{g(\omega)}{\lambda- i\omega}\di \omega,\quad \Re(\lambda)>0,\\
            \frac{k}{2}e^{-\lambda\tau}\int_{-\infty}^{+\infty}\frac{g(\omega)}{\lambda- i\omega}\di \omega+k\pi e^{-\lambda\tau}g(\frac{\lambda}{ i})\quad \Re(\lambda)<0.
        \end{cases}
    \end{equation}
\end{lemma}
\begin{proof}
    If $\lambda\in \mathbb{C}$ is a generalized eigenvalue associated with (\ref{eq:km4}), there exists some $\psi\in \text{Ran}(A(\lambda))$ such that
    \begin{equation}\label{eq:cha_eq1}
        \langle (\id-\frac{k}{2}e^{-\lambda\tau}\mathcal{P}^{\times}A(\lambda))\mathcal{P}^{\times}\psi|\varphi\rangle=0,
    \end{equation}
    for any $\varphi\in \mathrm{Exp}$. It follows from (\ref{eq:p_times}) that 
    $$
    \langle \mathcal{P}^\times A(\lambda)\mathcal{P}^\times\psi|\varphi\rangle=\langle\mathbb{I}|\varphi\rangle\langle A(\lambda)\mathcal{P}^\times\psi|\mathbb{I}\rangle =\langle \psi|\mathbb{I}\rangle\langle\mathbb{I}|\varphi\rangle \langle A(\lambda)\mathbb{I}|\mathbb{I}\rangle.
    $$
    Then, (\ref{eq:cha_eq1}) can be rewritten as 
    $$
    \begin{aligned}
        1&=\frac{k}{2}e^{-\lambda\tau}\langle A(\lambda)\mathbb{I}|\mathbb{I}\rangle\\
        &=\begin{cases}
            \frac{k}{2}e^{-\lambda\tau}\int_{-\infty}^{+\infty}\frac{g(\omega)}{\lambda- i\omega}\di \omega,\quad \Re(\lambda)>0,\\
            \frac{k}{2}e^{-\lambda\tau}\int_{-\infty}^{+\infty}\frac{g(\omega)}{\lambda- i\omega}\di \omega+k\pi e^{-\lambda\tau}g(\frac{\lambda}{ i})\quad \Re(\lambda)<0,
        \end{cases}
    \end{aligned}
    $$
    which completes the proof.
\end{proof}
Recall that in this section, the retarded resolvent operator of the delayed linear evolution equation (\ref{eq:km4}) $\Delta(\lambda)^{-1}:\mathcal{H}\to \mathcal{H}$ is given by (\ref{eq:km_retarded_resolvent}). Under hypotheses $\textbf{(H1-7)}$, we can define a corresponding generalized retarded resolvent operator $\mathcal{R}_\tau(\lambda):\kappa\mathrm{Exp}\to \mathrm{Exp}^{\prime}$ as 
\begin{equation}\label{eq:km_generalized_retarded_resolvent}
    \mathcal{R}_\tau(\lambda)=A(\lambda)\circ(\id-\frac{k}{2}e^{-\lambda\tau}\mathcal{P}^{\times}A(\lambda))^{-1}=(\id-\frac{k}{2}e^{-\lambda\tau}A(\lambda)\mathcal{P}^{\times})^{-1}\circ A(\lambda).
\end{equation}
The generalized spectrum can be defined as singularities of $\mathcal{R}_\tau(\lambda) i:\mathrm{Exp}\to\mathrm{Exp}^{\prime} $ in the fashion of Waelbroeck \cite{wael} as seen in Definition \ref{def:generalized_spectrum}. The following theorem concerns about distribution of the generalized spectrum associated with (\ref{eq:km4}).
\begin{theorem}\label{thm:km_resonance_pole}
Let $\mathcal{H}=L^2(\mathbb{R},g(\omega)\di \omega)$ and $g(\omega)=\delta_h(\omega-\omega_0):=\sqrt{\frac{h}{\pi}}e^{-h(\omega-\omega_0)^2}$. Several spectral properties associated with $\mathcal{R}_\tau(\lambda)$ on the Gelfand triple $\mathrm{Exp}\subset \mathcal{H}\subset\mathrm{Exp}^{\prime}$ can be concluded as\\
\noindent $\mathbf{(a)}$ The generalized spectrum consists of only generalized eigenvalues. In other words, we have
$$\hat{\sigma}_c(\mathcal{A})=\hat{\sigma}_r(\mathcal{A})=\emptyset;$$\\
\noindent $\mathbf{(b)}$ Let $\tau>0$, $\omega_0\not=0$ and $|k|<|k_c(\tau)|$, there exists some $h^\ast$ such that all of the generalized eigenvalues have only negative real parts for any $h>h^\ast$.  
\end{theorem}
\begin{proof}
    $\mathbf{(a)}$ directly follows from the compactness of $\mathcal{P}$ and Theorem \ref{thm:riesz schauder}. 
    
    For $\mathbf{(b)}$, it follows from Theorem \ref{thm:cplus} that generalized eigenvalues coincide with eigenvalues in the right half-plane. From Corollary \ref{coro:transition_point}, when $|k|<|k_c(\tau)|$, we have $\hat{\sigma}_p(\mathcal{A})\cap \{\lambda| \Re(\lambda)>0\}={\sigma}_p(\mathcal{A})\cap \{\lambda| \Re(\lambda)>0\}=\emptyset$. If $\lambda$ is a generalized eigenvalue with $\Re(\lambda)<0$, we have 
    \begin{equation}\label{eq:km_resonance_pole}
        1=e^{-\lambda\tau}\frac{k}{2}\sqrt{\frac{h}{\pi}}\int_{-\infty}^{+\infty}\frac{e^{-h(\omega-\omega_0)^2}}{\lambda- i\omega}\di \omega+k\sqrt{h\pi}e^{-\lambda\tau+h(\lambda- i\omega_0)^2}.
    \end{equation}
    Denote by $f(\lambda,h)$ the last term in (\ref{eq:km_resonance_pole})
    $$
    f(\lambda,h)=k\sqrt{h\pi}e^{-\lambda\tau +h(\lambda- i\omega_0)^2}.
    $$
    It follows that 
    $$
    f(\lambda,h)=k\sqrt{h\pi}A(\lambda,h)e^{h[\Re(\lambda)^2-(\Im(\lambda)-\omega_0)^2]},
    $$
    with $A(\lambda,h)=e^{-\lambda\tau+i(2h\Re(\lambda)(\Im(\lambda)-\omega_0))}$ and $|A(\lambda,h)|=e^{-\lambda\tau}$. The two lines $|\Re(\lambda)|=|\Im(\lambda- i\omega_0)|$ divide the complex plane into four sectors. When $|\Im(\lambda- i\omega_0)|>|\Re(\lambda)|$, $f(\lambda,h)$ converges to $0$ as $h\to\infty$. On the other hand, when $|\Im(\lambda- i\omega_0)|<|\Re(\lambda)|$, it is easy to verify that $f(\lambda,h)$ diverges as $h\to\infty$.

    Similar to the proof of Proposition \ref{prop:correspondance}, we can prove the set of generalized eigenvalues $\hat{\sigma}_p(\mathcal{\mathcal{A}})$ converges to a set 
    $$
    \{\lambda\in \mathbb{C}|G(\lambda)=1\ \text{and}\ |\Im(\lambda- i\omega_0)|>|\Re(\lambda)|\}
    $$
    in Hausdorff distance as $h\to \infty$.
    Recall that $G(\lambda)=\frac{k}{2}\frac{e^{-\lambda\tau}}{\lambda- i\omega_0}$. In particular, when $|k|<|k_c(\tau)|$, we can find infinite generalized eigenvalues $\{\lambda_p\}_{p\in\mathbb{N}}$ with $|\lambda_p|\to\infty$ for any sufficiently large $h$ ($h>h^\ast$).
\end{proof}
Recall that the complex order parameter $r(t)$ is defined by 
\begin{equation}\label{eq:km_complex_order_parameter}
r(t)=\int_{-\infty}^{+\infty}\int_0^{2\pi}e^{ i\theta}\rho(t,\omega,\theta)\delta_h(\omega-\omega_0)\di \theta\di \omega =(F_{-1}(t,\omega),\mathbb{I}).
\end{equation}
$F_{-1}(t,\omega)$ is the Fourier coefficient of $\rho(t,\omega,\theta)$ at $-1$-st position and its linear governing differential equation is given by (\ref{eq:km4}). Applying these properties of generalized spectrum defined within a Gelfand triple $\mathrm{Exp}\subset \mathcal{H}\subset{\mathrm{Exp}^\prime}$, we can prove the linear stability of incoherent solution $r(t)=0$ rigorously as following. 
\begin{theorem}\label{thm:linear_stability}
    When $|k|<|k_c(\tau)|$ and $\omega_0\not=0$, there exists $h^\ast$ such that the order parameter $r(t)$ is linearly stable for any $h>h^\ast$. Here, the linear stability refers that for any initial condition $f\in C([-\tau,0];\mathrm{Exp})$ with $f(0)=x$, $r(t)$ decays exponentially to zero as $t\to\infty$. 
\end{theorem}
\begin{proof}
    When $h$ is sufficiently large, we can find infinite generalized eigenvalues $\{\lambda_p\}_{p\in\mathbb{N}}\subset \{\lambda|\Re(\lambda)<0\}$ with $|\lambda_p|\to \infty$. For each $\lambda_p$, let $\gamma_p$ be a $C^{\infty}$ neighborhood which encloses $\lambda_p$. Firstly, we calculate the following integral (i.e., the generalized Riesz projection)
    \begin{equation}\label{eq:I_k}
        I_p=\frac{1}{2\pi i}\int_{\gamma_p}\langle \mathcal{R}_\tau(\lambda)\kappa(\psi)|\varphi\rangle\di \lambda,
    \end{equation}  
    where $\psi$ and $\varphi$ take values in $\text{Exp}$. Indeed, from (\ref{eq:km_generalized_retarded_resolvent}), 
    $$
    (\id-\frac{k}{2}e^{-\lambda\tau}A(\lambda)\mathcal{P}^\times)\circ\mathcal{R}_\tau(\lambda)\kappa(\psi)=A(\lambda)\kappa(\psi) 
    $$
    holds. It is easy to verify that
    $$
    \langle A(\lambda)\mathcal{P}^\times\mathcal{R}_\tau(\lambda)\kappa(\psi)|\mathbb{I}\rangle =\langle \mathcal{R}_\tau(\lambda)\kappa(\psi)|\mathbb{I}\rangle\langle A(\lambda)\mathbb{I}|\mathbb{I}\rangle.
    $$
    Hence, it follows that 
    $$
    \langle \mathcal{R}_\tau(\lambda)\kappa(\psi)|\mathbb{I}\rangle-\frac{k}{2}e^{-\lambda\tau}\langle \mathcal{R}_\tau(\lambda)\kappa(\psi)|\mathbb{I}\rangle\langle A(\lambda)\mathbb{I}|\mathbb{I}\rangle=\langle A(\lambda)\kappa(\psi)|\mathbb{I}\rangle
    $$
    and 
    \begin{equation}\label{eq:km_resolvent_simple_root}
    \begin{aligned}
    \langle \mathcal{R}_\tau(\lambda)\kappa(\psi)|\varphi\rangle&=\langle A(\lambda)\kappa(\psi)|\varphi\rangle+\frac{k}{2}e^{-\lambda\tau}\langle \mathcal{R}_\tau(\lambda)\kappa(\psi)|\mathbb{I}\rangle\langle A(\lambda)\mathbb{I}|\varphi\rangle\\
    &=\langle A(\lambda)\kappa(\psi)|\varphi\rangle+\frac{\langle A(\lambda)\mathbb{I}|\psi\rangle\langle A(\lambda)\mathbb{I}|\varphi\rangle}{\frac{2}{k}e^{-\lambda\tau}-\langle A(\lambda)\mathbb{I}|\mathbb{I}\rangle}.
    \end{aligned}
    \end{equation}
    It has been shown that when $h$ is sufficiently large, roots of $1-\frac{k}{2}e^{-\lambda\tau}\langle A(\lambda)\mathbb{I}|\mathbb{I}\rangle$(i.e. generalized eigenvalues) are all simple roots. For each generalized eigenvalue $\lambda$, there exists some generalized eigenfunction $\psi\in \text{Ran}(A(\lambda))\subset \mathrm{Exp}^\prime$ such that
    $$
    (\id-\frac{k}{2}e^{-\lambda\tau}A(\lambda)\mathcal{P}^\times)\psi=0.
    $$
    Equivalently, we have $\psi=\frac{k}{2}e^{-\lambda\tau}\langle \psi|\mathbb{I}\rangle \langle A(\lambda)\mathbb{I}|$. Hence, $\psi_p:=A(\lambda_p)\mathbb{I}$ is a generalized eigenfunction of $\lambda_p$. By substituting (\ref{eq:km_resolvent_simple_root}) into (\ref{eq:I_k}), we get
    \begin{equation}
        I_p=D_p \langle \psi_p|\psi\rangle\langle \psi_p|\varphi\rangle,
    \end{equation}
    where $D_p$ is the residue at $\lambda_p$ which is given by
    $$
    D_p=\lim_{\lambda\to\lambda_p}\frac{\lambda-\lambda_p}{\frac{2}{k}e^{\lambda\tau}-\langle A(\lambda)\mathbb{I}|\mathbb{I}\rangle}.
    $$
    It is easy to verify that 
    $$
    \lim_{p\to\infty}{\frac{|\Re(\lambda_p)|}{|\Im(\lambda_p)|}}=0.
    $$
    Hence, there exists some $\delta\in(0,\frac{\pi}{2})$ and $d<0$ such that a sector $\Gamma^\prime(l)=\Gamma_1(l)\bigcup\Gamma_2(l)$ where
    $$
    \begin{aligned}        &\Gamma_1(l)=\{z\in \mathbb{C}|z=d-re^{\pm i\delta},\ 0\leq r\leq l\},\\
        &\Gamma_2(l)=\{z\in \mathbb{C}|z=d-le^{ i\theta},\ -\delta\leq \theta\leq \delta\}\\
    \end{aligned}
    $$
    is defined on the left side of all generalized eigenvalues as $l\to\infty$. Such a sector is illustrated in Fig.~\ref{fig:4}.
    \begin{figure}
        \centering
        \includegraphics[width=0.5\linewidth]{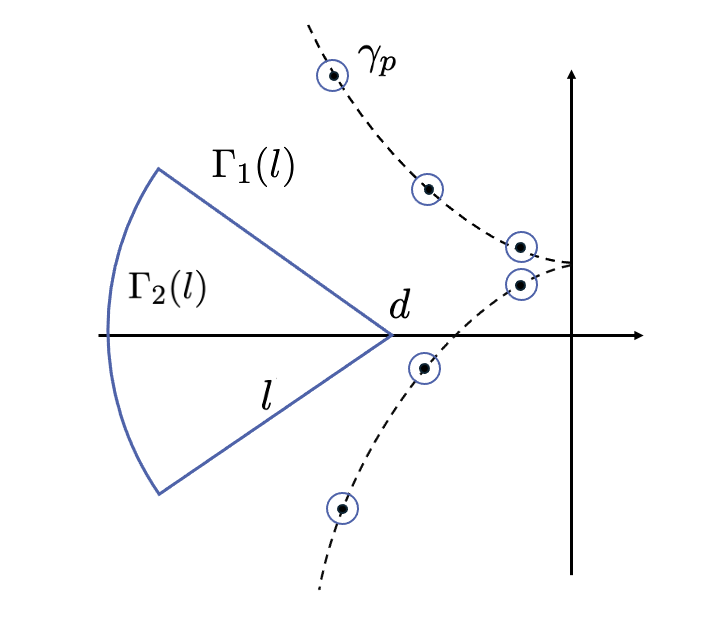}
        \caption{Since the continuous spectrum $\sigma_c(\mathcal{A})=i\mathbb{R}$ disappear in the generalized sense, the integral contour can be deformed to the left half-plane (i.e. $\Gamma_1(\infty)\cup\sum_{p\in \mathbb{N}}\gamma_p$)}.
        \label{fig:4}
    \end{figure}
    Now, we can obtain the following decomposition of integral
    $$
    \begin{aligned}
        \int_{\Gamma}e^{\lambda t}\langle \mathcal{R}_\tau(\lambda)\kappa(\psi)|\varphi\rangle\di \lambda&=(\int_{\Gamma_1}+\sum_{p\in\mathbb{N}}\int_{\gamma_p})e^{\lambda t}\langle \mathcal{R}_\tau(\lambda)\kappa(\psi)|\varphi\rangle\di\lambda\\
        &=\int_{\Gamma_1}e^{\lambda t}\langle \mathcal{R}_\tau(\lambda)\kappa(\psi)|\varphi\rangle\di\lambda+\sum_{p\in\mathbb{N}}D_p e^{\lambda_p t}\langle \psi_p|\psi\rangle\langle \psi_p|\varphi\rangle.
    \end{aligned}
    $$
    When $\Re(\lambda)<0$, we have  
    \begin{equation}\label{eq:jump}
        \langle A(\lambda)\kappa(\psi)|\varphi\rangle =((\lambda- i\mathcal{M})^{-1}\psi,\varphi^\ast)+2\pi\psi(\frac{\lambda}{ i})\varphi(\frac{\lambda}{ i})\delta_h(\frac{\lambda}{ i}).
    \end{equation}
    Since $\psi,\varphi\in \text{Exp}$, we can find positive constants $C_i$ and $\beta_i$ (i=1,2) such that
    $$
    |\psi(\lambda)|\leq C_1e^{\beta_1|\lambda|},\ |\varphi(\lambda)|\leq C_2e^{\beta_2|\lambda|}. 
    $$
    In this case, it follows from (\ref{eq:jump}) and (\ref{eq:km_resolvent_simple_root}) that
    $$
    \begin{aligned}
        \langle \mathcal{R}_\tau(\lambda)\kappa(\psi)|\varphi\rangle=\frac{1}{\frac{2}{k}e^{\lambda\tau}-\langle A(\lambda)\mathbb{I}|\mathbb{I}\rangle}M_\tau(\lambda),  
    \end{aligned}
    $$
    where 
    $$
    \begin{aligned}
        M_\tau(\lambda)&=\frac{2}{k}e^{\lambda\tau}((\lambda- i\mathcal{M})^{-1}\psi,\varphi^\ast)-((\lambda- i\mathcal{M})^{-1}\mathbb{I},\mathbb{I})((\lambda- i\mathcal{M})^{-1}\psi,\varphi^\ast)\\
        &\quad +((\lambda- i\mathcal{M})^{-1}\mathbb{I},\psi^\ast)((\lambda- i\mathcal{M})^{-1}\mathbb{I},\varphi^\ast)\\
        &\quad+2\pi \delta_h(\frac{\lambda}{ i})[\frac{2}{k}e^{\lambda\tau}\psi(\frac{\lambda}{ i})\varphi(\frac{\lambda}{ i})-((\lambda- i\mathcal{M})^{-1}\psi,\varphi^\ast)\\
       &\quad-\psi(\frac{\lambda}{ i})\varphi(\frac{\lambda}{ i})((\lambda- i\mathcal{M})^{-1}\mathbb{I},\mathbb{I})+ \psi(\frac{\lambda}{ i})((\lambda- i\mathcal{M})^{-1}\psi,\varphi^\ast)\\
        &\quad +\varphi(\frac{\lambda}{ i})((\lambda- i\mathcal{M})^{-1}\mathbb{I},\psi^\ast)].
    \end{aligned}
    $$
It is clear that when $\Re(\lambda)<0$ 
$$
||(\lambda- i\mathcal{M})^{-1}||=\text{dist}(\lambda, i\mathbb{R})^{-1}=\frac{1}{|\Re(\lambda)|}\to 0,
$$
as $|\Re(\lambda)|\to\infty$. Here $||\cdot||$ is the operator norm and $\text{dist}(\lambda,A)$ refers to the distance from $\lambda$ to the set $A$. 

Thus, for sufficiently large $|\Re(\lambda)|$, we can find positive constants $C_i$ (i=3,4,...,7) such that
$$
|M_\tau(\lambda)|\leq C_3+|\delta_h(\frac{\lambda}{ i})|(C_4+C_5 e^{\beta_1|\lambda|}+C_6 e^{\beta_2|\lambda|}+C_7 e^{(\beta_1+\beta_2)|\lambda|)}.
$$
Since 
$$|\frac{2}{k}e^{\lambda\tau}-\langle A(\lambda)\mathbb{I}|\mathbb{I}\rangle|=|\frac{2}{k}e^{\lambda\tau}-((\lambda- i\mathcal{M})^{-1}\mathbb{I},\mathbb{I})+\delta_h(\frac{\lambda}{ i})|\approx |\delta_h(\frac{\lambda}{ i})|$$ 
as $|\Re(\lambda)|\to \infty$, there exists $C_8>0$ such that
$$
\sup_{\lambda\in \Gamma_2(l)}|\langle \mathcal{R}_\tau(\lambda)\kappa(\psi)|\varphi\rangle|\leq C_8e^{(\beta_1+\beta_2)l}.
$$
Based on this estimate of $|\langle \mathcal{R}_\tau(\lambda)\kappa(\psi)|\varphi\rangle$, we can obtain 
$$
\begin{aligned}
    |\int_{\Gamma_2(l)}e^{\lambda t}\langle \mathcal{R}_\tau(\lambda)\kappa(\psi)|\varphi\rangle\di \lambda|&\leq \sup_{\lambda\in \Gamma_2(l)}|\langle \mathcal{R}_\tau(\lambda)\kappa(\psi)|\varphi\rangle| \int_{\Gamma_2(l)}|e^{\lambda t}|\di \lambda\\
    &\leq2C_8e^{(\beta_1+\beta_2)l+dt}\int_{0}^{\delta}le^{-lt\cos\theta}\di\theta\\
    &\leq2C_8 e^{(\beta_1+\beta_2)l+(d-l)t}\int_0^\delta le^{\frac{2}{\pi}lt\theta}\di\theta\\
    &=\frac{\pi}{t}C_8N(t,l)e^{dt},
\end{aligned}
$$
where
$$
N(t,l)=e^{l[\beta_1+\beta_2-(1-\frac{2}{\pi}\delta)t]}-e^{l(\beta_1+\beta_2-t)}.
$$
Then, for given $t>\frac{\pi}{\pi-2\delta}(\beta_1+\beta_2)$, we have
$$
\lim_{l\to\infty}\int_{\Gamma_2(l)}e^{\lambda t}\langle \mathcal{R}_\tau(\lambda)\kappa(\psi)|\varphi\rangle\di \lambda=0,
$$
and the spectral decomposition 
\begin{equation}\label{1}
    \int_{\Gamma}e^{\lambda t}\langle \mathcal{R}_\tau(\lambda)\kappa(\psi)|\varphi\rangle\di \lambda=\sum_{p\in\mathbb{N}}D_p e^{\lambda_p t}\langle \psi_k|\psi\rangle\langle \psi_k|\varphi\rangle
\end{equation}
holds. 

Put $f_\lambda:=\frac{k}{2}\mathcal{P}(\lambda-A_L)^{-1}f$ for $f\in C([-\tau,0];\text{Exp})$. We have
$$
|f_\lambda|=|\frac{k}{2}\int_{-\tau}^0e^{-\lambda(s+\tau)}\mathcal{P}(f(s))\di s|\leq M|\frac{e^{-\lambda \tau}-1}{\lambda}|,
$$
for some constant $M>0$. In particular, it is clear that because $\mathcal{P}$ is of finite rank, $f_\lambda\in \text{Exp}$ for any $\lambda\in \mathbb{C}$. Similar to the previous procedure, for any $\varphi\in \text{Exp}$ with $|\varphi(\lambda)|\leq E_1e^{\beta|\lambda|}$, we have
$$
    \begin{aligned}
        \langle \mathcal{R}_\tau(\lambda)\kappa(f_\lambda)|\varphi\rangle=\frac{1}{\frac{2}{k}e^{\lambda\tau}-\langle A(\lambda)\mathbb{I}|\mathbb{I}\rangle}\tilde{M}_\tau(\lambda),
    \end{aligned}
$$
and 
$$
\begin{aligned}
    |\tilde{M}_\tau(\lambda)|&\leq E_2+E_3e^{\tau|\lambda|}+|\delta_h(\frac{\lambda}{ i})|(E_4 e^{\tau|\lambda|}\\
    &\quad+E_5 e^{\beta|\lambda|}+E_6 e^{(\tau+\beta)|\lambda|}+E_7 e^{2\tau|\lambda|}),
\end{aligned}
$$
if $|\Re(\lambda)|$ is sufficiently large. The following estimate holds:
$$
\sup_{\lambda\in \Gamma_2(l)}|\langle \mathcal{R}_\tau(\lambda)\kappa(f_\lambda)|\varphi\rangle|\leq E_8 e^{\beta_3l},
$$
where $\beta_3:=\max\{2\tau,\tau+\beta\}$. Here $E_i$ ($i=1,2,...,8$), $\beta$ are positive constants. Likewise, it is easy to verify that there exists some $t_0$ such that if $t>t_0$, we can obtain 
$$
\lim_{l\to\infty}\int_{\Gamma_2(l)}e^{\lambda t}\langle \mathcal{R}_\tau(\lambda)\kappa(f_\lambda)|\varphi\rangle\di \lambda=0,
$$
and the following spectral decomposition holds:
\begin{equation}\label{2}
    \int_{\Gamma}e^{\lambda t}\langle \mathcal{R}_\tau(\lambda)\kappa(f_\lambda)|\varphi\rangle\di \lambda=\sum_{p\in\mathbb{N}}D_p e^{\lambda_p t}\langle \psi_k|f_{\lambda_p}\rangle\langle \psi_p|\varphi\rangle.
\end{equation}
When $u(t)$ solves (\ref{eq:km4}) under initial condition $(x,f)\in \mathrm{Exp}\times C([-\tau,0];\mathrm{Exp})$, I
it follows from (\ref{eq:km_complex_order_parameter}) that $r(t)=(u(t),\mathbb{I})$. We can get
\begin{equation}\label{eq:final1}
(u(t),\mathbb{I})=\frac{1}{2\pi  i}\lim_{b\to\infty}\int_{a- ib}^{a+ ib}e^{\lambda t}[(\Delta(\lambda)^{-1}x,\mathbb{I})+(\Delta(\lambda)^{-1}f_\lambda,\mathbb{I})]\di\lambda, 
\end{equation}
for some $a>0$. The entire integral contour lies on the right half-plane. It follows from Proposition \ref{pro:generalized_resolvent_continuation} that when $\Re(\lambda)>0$
$$
(\Delta(\lambda)^{-1}\psi,\varphi)=\langle \mathcal{R}_\tau(\lambda)\kappa(\psi),\varphi\rangle.
$$
Hence, (\ref{eq:final1}) can be rewritten as 
\begin{equation}
    (u(t),\mathbb{I})=\frac{1}{2\pi  i}\lim_{b\to\infty}\int_{a- ib}^{a+ ib}e^{\lambda t}[\langle\mathcal{R}_\tau(\lambda)\kappa(x),\mathbb{I}\rangle+\langle\mathcal{R}_\tau(\lambda)\kappa(f_\lambda),\mathbb{I}\rangle]\di\lambda.
\end{equation}
When $|k|<|k_c(\tau)|$ and $h$ is sufficiently large, by applying (\ref{1}) and (\ref{2}), we can obtain
$$
(u(t),\mathbb{I})=\sum_{p\in \mathbb{N}} D_p e^{\lambda_p t}\langle \psi_p|x+f_{\lambda_p}\rangle\langle \psi_p|\mathbb{I}\rangle,
$$
which converges to zero exponentially as $t\to\infty$. 
\end{proof}

\noindent\textbf{Acknowledgment.}
The author wishes to thank Hayato Chiba for fruitful discussions and valuable suggestions on the paper. 
The author H.S. acknowledges supports of JST Moonshot Research and Development Grant Number JPMJMS2023, Japan.



\begin{thebibliography}{GMS13}

\bibitem[BP01]{batkai1}
	\textsc{A. Bátkai, S. Piazzera}:
	\textit{Semigroups and linear partial differential equations with delay},
	J. Math. Anal. Appl, \textbf{264} (2001), 1--20.

\bibitem[BP05]{batkai2}
	\textsc{A. Bátkai, S. Piazzera}:
	\textit{Semigroups for delay equations},
	A K Peters/CRC Press, Cambridge, 2005.

\bibitem[B69]{bruyn}
	\textsc{G. F. C. de Bruyn}:
	\textit{The existence of continuous inverse operators under certain conditions},
	J. London Math. Soc., \textbf{44} (1969), 68--70.

\bibitem[CH89]{ch}
	\textsc{J. D. Crawford, P. D. Hislop}:
	\textit{Application of the method of spectral deformation to the Vlasov-Poisson system},
	Ann. Physics, \textbf{189} (1989), 265--317.

\bibitem[CGHJK96]{CGHJK}
	\textsc{R. M. Corless, G. H. Gonnet, D. E. G. Hare, D. J. Jeffery, D. E. Knuth}:
	\textit{On the Lambert W function},
	Adv Comput Math, \textbf{5} (1996), 329--359.

\bibitem[C15a]{chibaA}
	\textsc{H. Chiba}:
	\textit{A proof of the Kuramoto conjecture for a bifurcation structure of the infinite dimensional Kuramoto model},
	Ergo. Theo. Dyn. Syst, \textbf{35} (2015), 762--834.

\bibitem[C15b]{chibaB}
	\textsc{H. Chiba}:
	\textit{A spectral theory of linear operators on rigged Hilbert spaces under analyticity conditions},
	Adv. in Math, \textbf{273} (2015), 324--379.

\bibitem[CM21]{medvedev}
	\textsc{H. Chiba, G. Medvedev}:
	\textit{Stability and bifurcation of mixing mixing in the Kuramoto model with inertia},
	SIAM J. Math. Anal, \textbf{54} (2021), 1797--1819.

\bibitem[D95]{davies}
	\textsc{E. B. Davies}:
	\textit{Spectral theory and differential operators},
	Cambridge University Press, Cambridge, 1995.

\bibitem[DHJ21]{ha}
	\textsc{Z. Ding, S. Y. Ha, S. Jin}:
	\textit{A local sensitivity analysis in Landau damping for the kinetic Kuramoto equation with random inputs},
	Quart. Appl. Math, \textbf{79} (2021), 229--264.

\bibitem[E65]{Edw}
	\textsc{R. E. Edwards}:
	\textit{Functional analysis. Theory and applications},
	Holt, Rinehart and Winston, New York, 1965.

\bibitem[EN00]{nagel}
    \textsc{K. J. Engel, R. Nagel}:
	\textit{One-parameter semigroups for linear evolution equations},
	  Springer New York, NY, 2013.

\bibitem[FGG16]{fernandez}
	\textsc{B. Fernandez, D. Gérard-Varet, G. Giacomin}:
	\textit{Landau Damping in the Kuramoto model},
	Ann. Henri Poincaré, \textbf{17} (2016), 1793--1823.

\bibitem[GP11]{gp}
	\textsc{M. Gadella, G. P. Pronko}:
	\textit{The Friedrichs model and its use in resonance phenomena},
	Fortschr. Phys., \textbf{59} (2011), 795--859.

\bibitem[HMZ12]{hmz}
	\textsc{S. Hejazian, M. Mirzavaziri, O. Zabeti}:
	\textit{Bounded operators on topological vector spaces and their spectral radii},
	Filomat, \textbf{26} (2012), no. 6, 1283--1290.

\bibitem[HS96]{hislop}
	\textsc{P. D. Hislop, I. M. Sigal}:
	\textit{Introduction to spectral theory. With applications to Schrödinger operators},
	Springer-Verlag, New York, 1996.

\bibitem[K75]{kuramoto}
	\textsc{Y. Kuramoto}:
	in \textit{Proceedings of the international symposium on mathematical problems in theoretical physics},
	edited by H. Araki, Lecture Notes in Physics Vol. 39 (Springer, Berlin, 1975); \textit{Chemical oscillations, waves, and turbulence} (Springer, Berlin, 1984)

\bibitem[K95]{kato}
	\textsc{T. Kato}:
	\textit{Perturbation theory for linear operators},
	Springer-Verlag, Berlin, 1995.
  
\bibitem[M61]{maeda}
	\textsc{F. Maeda}:
	\textit{Remarks on spectra of operators on a locally convex space},
	Proc. Nat. Acad. Sci. U.S.A., \textbf{47} (1961).

\bibitem[N88]{nakagiri2}
	\textsc{S. Nakagiri}:
	\textit{Structural properties of functional differential equations in Banach spaces},
	Osaka J. Math, \textbf{25} (1988), 353--398.

\bibitem[N16]{nishi}
	\textsc{J. Nishiguchi}:
	\textit{On parameter dependence of exponential stability of equilibrium solutions in differential equations with a single constant delay},
	Discrete and Continuous Dynamical Systems, \textbf{36} (2016).

\bibitem[N24]{noll}
	\textsc{D. Noll}:
	\textit{Topological spaces satisfying a closed graph theorem},
	Topology and its Applications, \textbf{349} (2024).


\bibitem[R57]{ringrose}
	\textsc{J. R. Ringrose}:
	\textit{Precompact linear operators in locally convex spaces},
	Mathematical Proceedings of the Cambridge Philosophical Society, \textbf{53} (1957), 581--591.

\bibitem[RS78]{reed}
    \textsc{M. Reed, S. Barry}:
	\textit{Method of modern mathematical physics IV. Analysis of operators},
	  Academic Press, New York-London, 1978.

\bibitem[R80]{rauch}
	\textsc{J. Rauch}:
	\textit{Perturbation theory for eigenvalues and resonances of Schrödinger Hamiltonians},
	J. Funct. Anal., \textbf{35} (1980), no. 3, 304--315.

\bibitem[SM91]{sm}
	\textsc{S. H. Strogatz, R. E. Mirollo}:
	\textit{Stability of incoherence in a population of coupled oscillators},
	J Stat Phys, \textbf{63} (1991), 613--635.

\bibitem[SMM92]{smm}
	\textsc{S. H. Strogatz, R. E. Mirollo, P. C. Matthews}:
	\textit{Coupled nonlinear oscillators below the synchronization threshold: relaxation be generalized Landau damping},
	Phys. Rev. Lett, \textbf{68} (1992), no. 18, 2730--2733.

\bibitem[SW99]{SW}
    \textsc{H. H. Schaefer, M. P. Wolff}:
	\textit{Topological vector spaces},
	  Springer New York, 1999.

\bibitem[S00]{strogatza}
	\textsc{S. H. Strogatz}:
	\textit{From Kuramoto to Crawford: exploring the onset of synchronization in populations of coupled oscillators},
	Phys. D, \textbf{143} (2000), no. 1-4, 1--20.

\bibitem[T67]{Treves}
    \textsc{F. Trèves}:
	\textit{Topological vector spaces, distributions and kernels},
	  Academic Press, New York-London, 1967.

\bibitem[WD18]{wd}
	\textsc{H. Wu, M. Dhamala}:
	\textit{Dynamics of Kuramoto oscillators with time-delayed positive and negative couplings},
	Phys. Rev. E, \textbf{98} (2018), 032221.

\bibitem[W58]{wael}
    \textsc{L. Waelbroeck}:
	\textit{Locally convex algebras: spectral theory},
	  Seminar on complex analysis at Institute of Advanced Study, 1958.

\bibitem[YS99]{yeung}
	\textsc{M. K. S. Yeung, S. H. Strogatz}:
	\textit{Time delay in the Kuramoto model of coupled oscillators},
	Phys. Rev. Lett., \textbf{82} (1999), 648--651.

\bibitem[ZH98]{zumbrun}
	\textsc{K. Zumbrun, P. Howard}:
	\textit{Pointwise semigroup methods and stability of viscous shock waves},
	Indiana Univ. Math. J, \textbf{47} (1998), 741--871.

\end{thebibliography}
\end{document}